\documentclass[12pt]{amsart}
\usepackage[utf8]{inputenc}

\title[Compressed sensing photoacoustic tomography]{Recovery guarantees for compressed sensing photoacoustic tomography}
\author[A.\ Felisi]{Alessandro Felisi}
\address{Institute for Applied Mathematics (IAM), University of Bonn, Endenicher Allee 60, 53115 Bonn, Germany.}
\email{afelisi@uni-bonn.de}

\date{\today}

\usepackage{amsmath, amssymb, amsfonts, amsthm, amsopn, cancel}
\usepackage{graphicx,tikz}
\usepackage[all]{xy}
\usepackage{amsmath}
\usepackage{bbm}
\usepackage{multirow, longtable, makecell, caption, array, enumitem}
\usepackage{amssymb}
\usepackage{mathtools}
\mathtoolsset{showonlyrefs}
\usepackage{bookmark}
\usepackage{hyperref}
\usepackage[font=small,labelfont=bf]{caption}
\usepackage{esint}

\allowdisplaybreaks

\calclayout

\theoremstyle{plain}
\newtheorem{theorem}{Theorem}[section]
\newtheorem{lemma}[theorem]{Lemma}

\newtheorem{proposition}[theorem]{Proposition}

\theoremstyle{definition}
\newtheorem{assumption}[theorem]{Assumption}
\theoremstyle{definition}

\theoremstyle{definition}

\theoremstyle{plain}
\newtheorem*{theorem*}{Theorem}
\theoremstyle{definition}
\newtheorem*{definition*}{Definition}
\newtheorem{example}[theorem]{Example}
\newtheorem{exampleoutlook}[theorem]{Example/Outlook}
\theoremstyle{remark}
\newtheorem{remark}[theorem]{Remark}

\def\bC{\mathbb{C}}

\def\bN{\mathbb{N}}

\def\bR{\mathbb{R}}

\def\bZ{\mathbb{Z}}

\def\cD{\mathcal{D}}
\def\cF{\mathcal{F}}
\def\cH{\mathcal{H}}

\def\cL{\mathcal{L}}
\def\cM{\mathcal{M}}

\def\cP{\mathcal{P}}

\def\lc{\left(}
\def\rc{\right)}

\def\supp{\operatorname{supp}}

\def\*b{*_{\bullet}}

\def\Bd'{B_{\delta'}}

\def\cBd'{\bar{B}_{\delta'}}

\newcommand{\diam}{\mathrm{diam}}
\newcommand{\diff}{\mathrm{d}}

\newcommand{\Span}{\mathrm{span}}

\begin{document}
\begin{abstract}
Photoacoustic tomography is an emerging medical imaging technology whose primary aim is to map the high-contrast optical properties of biological tissues by leveraging high-resolution ultrasound measurements. Mathematically, this can be framed as an inverse source problem for the wave equation over a specific domain. In this work, for the first time, it is shown how, by assuming signal sparsity, it is possible to establish rigorous stable recovery guarantees when the data collection is given by spatial averages restricted to a sparse, randomly sampled portion of the boundary. Our framework encompasses many approaches that have been considered in the literature. The result is a consequence of a general framework for subsampled inverse problems developed in previous works and refined stability estimates for an inverse problem for the wave equation with surface measurements.
\end{abstract}

\subjclass{35R30}

\maketitle


\section{Introduction}
Photoacoustic tomography (PAT) is an emerging medical imaging technology, recognized as one of the most sophisticated among hybrid imaging modalities \cite{agranovsky2007, kuchment2007, wang2007, wang2009, tuchin2016}. PAT integrates two distinct forms of energy: electromagnetic (EM) waves, such as light, and acoustic waves, like ultrasound. The core objective of PAT is to map the optical properties of biological tissues, which is particularly valuable for applications such as tumor detection, vascular imaging, and even monitoring oxygen saturation levels in tissues.

The underlying mechanism of PAT is based on the thermoacoustic effect. When tissue is exposed to a brief pulse of EM radiation, it undergoes rapid heating and subsequent thermal expansion. This expansion creates a pressure wave that propagates through the tissue and can be detected by wide-band ultrasonic transducers placed outside the body. In the context of tumor detection, for instance, the imaging modality relies on the fact that cancerous tissues absorb a higher level of EM radiation than healthy tissues, and thus generate stronger pressure waves. By reconstructing the initial distribution of pressure, PAT can produce highly informative data that reveals critical diagnostic details, making it a powerful tool in early disease detection and ongoing medical research.

\subsection{Mathematical model and stability issues}
In the literature, the following simple model for the pressure wave dynamics, known as the \textit{free space model}, is often considered. We assume that the pressure wave $u$ is generated by an initial pressure $u_0$, which is supported within a compact set $K\subset \bR^3$. The wave dynamics is described by the wave equation:
\begin{equation}
\label{eq:wav-IVP0}
    \begin{cases}
    u_{tt} - c^2\Delta u = 0, & \text{in } \mathbb{R}^3 \times [0, +\infty), \\
    u(\cdot, 0) = u_0, & \text{in } \mathbb{R}^3, \\
    u_t(\cdot, 0) = 0, & \text{in } \mathbb{R}^3,
\end{cases}
\end{equation}
where $c=c(x)$ represents the speed of propagation in the biological tissue. Here $u_{t},u_{tt}$ denote $\partial_t u,\partial^2_{t}u$, respectively. The pressure wave $u$ is measured on a (smooth) \textit{acquisition surface} $\Sigma\subset\bR^3$ over a finite time interval $[0, T]$, with $T < +\infty$. Additionally, we assume that $K$ is located at a positive distance from $\Sigma$. The reconstruction problem in PAT then reduces to an inverse problem for the wave equation: given $u|_{\Sigma \times [0, T]}$, reconstruct the initial pressure $u_0$ in a stable manner.

In this paper, we focus on the case where the speed of propagation $c$ is constant, and for simplicity, we assume $c \equiv 1$. For more detailed studies addressing the case of non-constant $c$, we refer the reader to \cite{agranovsky2, stefanov2009, stefanov2015, belhachmi2016, stefanov2016, tick2020, mathison2019}. Additionally, we highlight an alternative and possibly more accurate model, known as the \textit{bounded space model} \cite{cox2007, ammari2010, kunyansky2013, acosta2015, holman2015, stefanov2015, chervova2016, stefanov2016, alberti2018, alberti2020}, which incorporates the interaction of the transducers with the pressure wave through appropriate boundary conditions.

The ill-posedness of PAT is known to critically depend on the relationship between $K$, the compact set containing the support of the initial datum $u_0$, and the acquisition surface $\Sigma$. Roughly speaking, if $\Sigma$ is sufficiently large, the reconstruction problem becomes well-posed; otherwise, it remains ill-posed. This statement can be formalized using the following \textit{visibility condition} for constant-speed waves.

\begin{definition*}[\cite{nguyen2011}]
Let $\Sigma \subset \bR^3$ be a surface and let $\Omega_0 \Subset\bR^3$ be an open subset with $d(\Omega_0,\Sigma)>0$. We say that the pair $(\Omega_0,\Sigma)$ satisfies the \textit{visibility condition} if the following does not hold: there exists $ \xi \in \mathbb{R}^3 \setminus \{0\} $ and open sets $ \Omega_K, \Omega_{\Sigma} \subset \mathbb{R}^3 $ with $ \Omega_K \subset \Omega_0 $ and $ \Omega_{\Sigma} \supset \Sigma $ such that, for every $ x \in \Omega_K $, the line $ \{x + t\xi : t \in \mathbb{R}\} $ does not intersect $ \Omega_{\Sigma} $.
\end{definition*}
See Fig.~\ref{fig:violation} for an example of a situation where the visibility condition is not satisfied.
\begin{figure}[t!]
    \centering
    \includegraphics[width=0.9\textwidth]{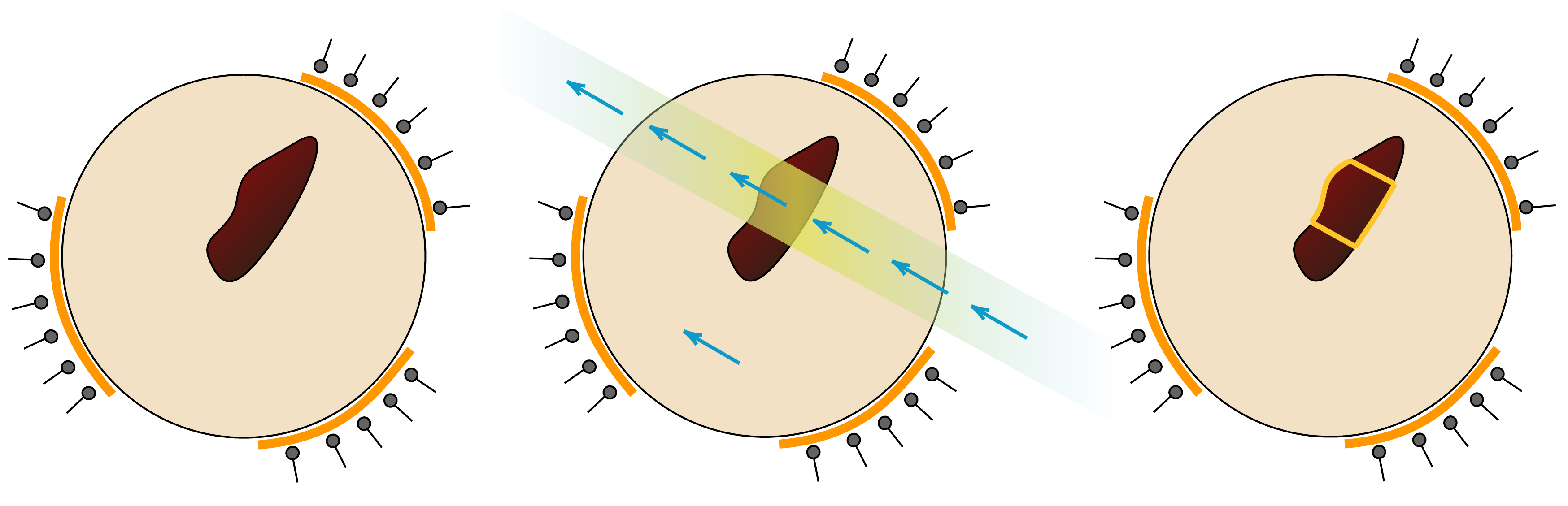}
    \put(-330,73){\textcolor{white}{\scalebox{0.8}{$\Omega_{0}$}}}
    \put(-370,50){$\Sigma$}
    \put(-314,28){$\Sigma$}
    \put(-305,82){$\Sigma$}
    \put(-212,43){$\xi$}
    \put(-59,74){\textcolor{white}{\scalebox{0.7}{$\Omega_{K}$}}}
    \put(-336,-5){\scalebox{0.8}{(a)}}
    \put(-205,-5){\scalebox{0.8}{(b)}}
    \put(-65,-5){\scalebox{0.8}{(c)}}
    \caption{Example of failure of the visibility condition. (a) The set $\Omega_0$ and the acquisition surface $\Sigma$. (b) A nonzero element $\xi$ and the associated line bundle that does not intersect a neighbourhood $\Omega_{\Sigma}$ of $\Sigma$. (c) An open set $\Omega_K$ contained in the intersection of $\Omega_0$ and the line bundle.}
    \label{fig:violation}
\end{figure}
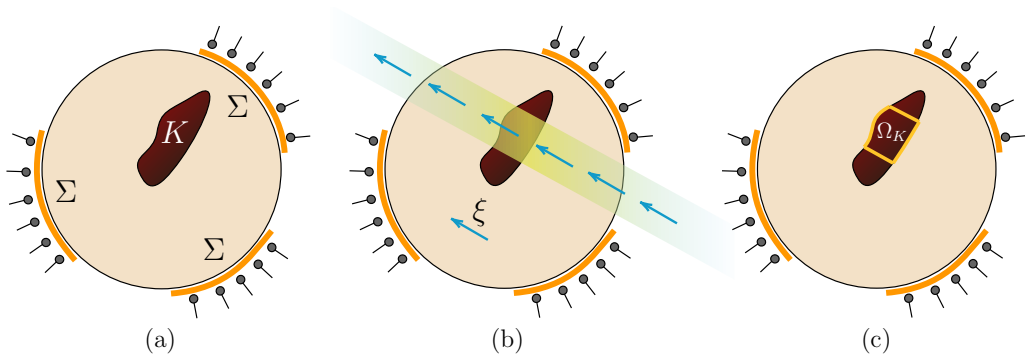

The following is a negative result, demonstrating that, if the visibility condition is not satisfied, then conditional Hölder stability cannot be achieved with respect to any Sobolev norm.

\begin{theorem*}[{\cite[Theorem 3.1]{nguyen2011}}]
Let $ U $ be the linear operator defined by
\[
    U\colon L^2(\Omega_0) \rightarrow \mathcal{D}'(\Sigma \times [0, T]),\footnote{Here $\cD'$ is the space of distributions on $\Sigma\times[0,T]$, defined as the dual of $C^{\infty}_c(\Sigma\times[0,T])$.}
\]
mapping the initial datum $ u_0 $ to the restriction to $ \Sigma \times [0, T] $ of the corresponding solution of the wave equation with constant velocity $ c $.

Suppose that $(\Omega_0,\Sigma)$ does not satisfy the visibility condition. Then there do not exist constants $ \mu, \delta, C > 0 $ and $ s_0, s_1 \geq 0 $ such that
\[
    \|u_0\|_{L^2(\Omega_0)} \leq C \|U u_0\|_{H^{s_0}(\Sigma \times [0, T])}^\mu, \quad \forall u_0 \in H^{s_1}(\Omega_0) \text{ with } \|u_0\|_{H^{s_1}(\Omega_0)} \leq \delta.
\]
\end{theorem*}

We also highlight the positive result in \cite[Theorem 3]{stefanov2009}, which states that a visibility condition similar to our definition is sufficient for Lipschitz stable reconstruction with respect to appropriate Sobolev norms.

The visibility condition is closely related to the orientation of the singularities of $ u_0 $, due to the propagation of singularities for the wave equation. Roughly speaking, the acquisition surface $ \Sigma $ must cover all directions orthogonal to the singularities to enable stable reconstruction of $ u_0 $. Unfortunately, this condition might not be satisfied in practical scenarios with typical transducer configurations. This raises a natural question: could a different prior on the initial datum $ u_0 $ lead to a stable recovery result for subsampled data?

\subsection{Main contribution}
The aim of this work is to obtain a stable recovery result (see Theorem~\ref{thm:main_result} for the precise statement) for the reconstruction problem in sparse data PAT using a wavelet sparse prior. We mention that, although our result provides a stable recovery of a signal from the observation of wave propagation on a randomly subsampled portion of a closed surface with no reference to geometric visibility conditions on the sampled detector subset, our setting is nevertheless different in spirit from limited-view PAT, where some directions are assumed to be invisible from detectors. We provide a brief overview of our result as follows.

We assume that the datum $ u_0 \in L^2(B_1) $ is $ s $-sparse with respect to a suitable wavelet orthonormal basis $ (\phi_{j,n})_{(j,n) \in \Gamma} $, $\Gamma$ being a suitable index set, -- see Sec.~\ref{sec:wavelets} --, namely there exists $ S \subset \Gamma $ with $ |S| \leq s $ such that
\begin{equation}
\label{eq:u0definition}
    u_0 = \sum_{(j,n) \in S} x_{j,n} \phi_{j,n}
\end{equation}
for some coefficients $ x_{j,n} \in \mathbb{C} $. Furthermore, we assume that the maximum scale $ j $ appearing in the above expansion is $ j_0 \in \mathbb{N} $.

Let $ u $ be the pressure wave associated with the initial datum $ u_0 $. Let $\Sigma\coloneqq\partial\Omega$ be the boundary of a smooth bounded domain $\Omega\subset\bR^3$. We assume that the following stability estimates for the complete data on $\Sigma$ hold for some sufficiently large $T>0$:
\begin{equation}
\label{eq:stab_rakesh}
\begin{split}
    c\|u_0\|_{L^2(\bR^3)} \leq &\|u\|_{L^2(0,T;L^2(\Sigma))} \leq
    C\|u_0\|_{L^2(\bR^3)},\\
    &\|u\|_{L^2(0,T;H^1(\Sigma))} \leq
    C\|u_0\|_{H^1(\bR^3)},
\end{split}
\end{equation}
for every smooth $u_0$ supported in $\Omega_0$, where $\Omega_0$ is an open set containing the supports of all wavelets $(\phi_{j,n})_{(j,n)\in\Gamma}$ such that $\Omega_0\Subset\Omega$, so that $d(\Omega_0,\Sigma)>0$. We will prove in Sec.~\ref{sec:spherical} that \eqref{eq:stab_rakesh} is satisfied in the case where $\Omega$ is strictly convex.

Instead of measuring $ u $ on the entire surface $ \Sigma $, we proceed as follows: we partition $ \Sigma $ into a finite number of detectors $ (E_i)_{i=1}^N $; the measurements are taken as a noisy version of local averages of $ u $ on only $ m $ randomly chosen detectors $ E_{i_1}, \dots, E_{i_m} $, where $ m \leq N $. More precisely, we assume that $ \varepsilon_k = \varepsilon_k(t) $ represents the noise on each detector $ E_{i_k} $, and that the measurements are given by the time-dependent functions
\begin{align*}
    y_k(t) &\coloneqq \frac{1}{|E_{i_k}|} \int_{E_{i_k}} u(y,t)\, \mathrm{d}\sigma(y) + \varepsilon_k(t),\qquad t\in[0,T],
\end{align*}
where $ |\cdot| $ and $\diff{\sigma}$ denote the surface measure on $\Sigma$. We assume that the noise is controlled a priori by some noise level $ \beta \geq 0 $, i.e., that $ \|\varepsilon_k\| \leq \beta $ with respect to some appropriate norm $ \|\cdot\| $. The detectors used in the reconstruction procedure are chosen according to a random design -- see Section~\ref{sec:detectors}.

We now state a simplified version of the main result, given by Theorem~\ref{thm:main_result}.

\begin{theorem*}
Suppose the detectors $ (E_i)_{i=1}^N $ satisfy the uniform eccentricity condition -- see \eqref{eq:ass_detectors} --, and that
\begin{equation}
\label{eq:detector-condition0}
    \text{diam}_{\mathbb{R}^3}(E_i) \lesssim 2^{-j_0}.
\end{equation}
Suppose also that
\begin{align*}
    m \gtrsim j_0 s \cdot (\emph{log factors}).
\end{align*}
Let $ \widehat{u} $ be a solution of the following $ \ell^1 $-minimization problem:
\begin{equation}
    \min_{v_0} \sum_{j,n} |\langle v_0,\phi_{j,n} \rangle|\colon\qquad
    \frac{1}{m}\sum_{k=1}^m \bigg\|\frac{1}{|E_{i_k}|}\int_{E_{i_k}} v\,\diff{\sigma}-y_k\bigg\|_{L^2(0,T)}^2 \leq C\beta^2,
\end{equation}
where $C>0$ is a suitable constant and the minimum is taken over $$\Span\{\phi_{j,n}\colon\ (j,n)\in\Gamma,\, j\leq j_0 \}$$ and $v$ is the solution of \eqref{eq:wav-IVP0} with initial datum $v_0$ and $c\equiv 1$. Then, with overwhelming probability, the following recovery estimate holds:
\begin{equation}
    \|u_0 - \widehat{u}\|_{L^2} \lesssim \beta.
\end{equation}
\end{theorem*}

In other words, the result states that we can achieve Lipschitz stable recovery of $ s $-sparse signals, provided the detectors are small enough relative to the signal resolution (specifically, relative to the highest wavelet scale index $ j_0 $ appearing in the expansion \eqref{eq:u0definition}), and the number of sampled detectors is proportional to $ j_0 s $, up to logarithmic factors. In Theorem~\ref{thm:main_result}, we also account for a possible sparsity defect and infinite resolution of $ u_0 $ by introducing appropriate sparsity and truncation errors. We also derive an extension of this result -- see Theorem~\ref{thm:new-result} -- that can deal with more general sensing patterns.

\begin{remark}
If the sparsity level is sufficiently small, the number of detectors $m$ used in the reconstruction is strictly smaller than the number of detectors $N$ in the partition. Indeed, condition \eqref{eq:detector-condition0} implies that
\begin{equation}
\label{eq:number_total_detectors}
    N \gtrsim 2^{2j_0}.
\end{equation}
On the other hand, up to log factors, the number of detectors used in the reconstruction is given by
\begin{align*}
    m \gtrsim j_0 s.
\end{align*}
Therefore, our result shows that it is possible to achieve stable recovery of sparse signals in the sparse data regime of PAT. See also Example~\ref{ex:besov} for explicit estimates under a Besov space prior. We remark that each measurement $y_k=y_k(t)$ corresponds to a time-dependent function; in practice, such measurements are approximated via suitable time discretizations $\{y_k(t_l)\}_{l=1,\dots,N_t}$; although important, we won't discuss matters concerning this additional discretization in what follows.
\end{remark}

\subsection{Comparison with the literature}
Obtaining stable recovery guarantees for $\ell^1$-minimization relies on tools from compressed sensing (CS), a class of signal processing techniques that exploits signal sparsity to achieve stable reconstruction with few samples \cite{CRT,donoho,foucartrauhut}. Since the advent of CS, numerous efforts have been made to develop CS-based strategies to accelerate the acquisition process in PAT \cite{provost2009,meng2012,haltmeier2016,sandbichler2015,burgholzer2016,arridge2016,betcke2017,alberti2021,haltmeier2017,antholzer2019,mathison2019,campodonico2021,haltmeier2024}. Most cited works consider the following general measurement model:
\begin{equation}
\label{eq:type-of-meas}
    y_k(t) = \langle u(\cdot,t),\psi_k \rangle_{L^2(\Sigma)},
\end{equation}
where $\psi_k \in L^2(\Sigma)$ are suitable masks. In the literature, $\psi_k$ are typically binary masks, i.e., functions on $\Sigma$ taking values in $\{0,1\}$ or $\{\pm 1\}$; this is similar in spirit to the type of measurements considered in this work. In some cases, the measurements are given by pointwise samples $u(x_k,\cdot)$ at specific locations $x_1,\dots,x_m\in\Sigma$; such measurements can be approximated by \eqref{eq:type-of-meas} choosing $\psi_k$ as rescaled indicator functions with small support, effectively providing local averages of $u(x,\cdot)$ around $x_k$.

The practical realization of these measurements can be achieved using different scanning systems, such as the Fabry-P\'erot photoacoustic scanner \cite{Zhang2008} or mechanically scanned piezoelectric detectors. There are essentially two classes of scanning patterns. In \textit{single-point scanning}, the masks $\psi_k$ are localized in small regions concentrated around specific points on a grid -- see, for instance, \cite{provost2009,meng2012}. To reduce coherence in the measurements, the mask locations are often chosen according to a random design. This more conventional approach is the focus of our method in Sec.~\ref{sec:main-result}. On the other hand, in \textit{pattern-interrogation scanning}, a series of orthogonal masks $\psi_k$ supported over the entire surface $\Sigma$ is used, rather than focusing on specific locations. This approach speeds up acquisition in many scenarios and is adopted, for example, in \cite{haltmeier2016,sandbichler2015,arridge2016,betcke2017,haltmeier2024}. We address this method in Sec.~\ref{sec:new-result}.

Reconstruction procedures using CS in PAT typically fall into two categories. The first, known as the \textit{one-step approach}, aims to reconstruct the unknown signal $u_0$ directly from the measurements $y_k$ -- see, for instance, \cite{provost2009,meng2012,arridge2016,betcke2017}. The second, the \textit{two-step approach}, first applies compressed sensing to reconstruct the full data $u|_{\Sigma \times [0,T]}$, which is then used for signal reconstruction via standard techniques such as time-reversal or back-projection. In the latter case, applying CS techniques requires enforcing sparsity in the measurements $y_k$; this is typically achieved by processing the data with suitable sparsifying transformations---see, for example, \cite{haltmeier2016,sandbichler2015}. As argued in \cite{haltmeier2016}, the two-step approach has numerical advantages, as it avoids solving the wave equation in the first step, thereby reducing computational cost. However, it relies on the assumption that boundary data can be efficiently sparsified via an appropriate transformation.

Due to the complexity of the forward map in PAT, the interaction between the signal’s sparsity and the sparsity of the measurements has so far defied a rigorous theoretical understanding. As a result, most existing research remains empirical, providing only partial answers regarding theoretical guarantees. In this work, we present, to the best of the author's knowledge, the first rigorous theoretical guarantees for PAT in the sparse data regime, leveraging the sparsity of the signal itself with respect to a suitable wavelet basis. We also mention the work of \cite{campodonico2021}, which shares a similar spirit with our study. There, the authors show that in many experimental setups, CS in PAT reduces to the well-known problem of reconstructing a sparse signal from a few samples of its Fourier coefficients. While this is of theoretical interest, \cite{campodonico2021} does not provide recovery guarantees such as those in Theorem~\ref{thm:main_result}, and the performance of the proposed approach in PAT is only assessed via numerical simulations.

\subsection{Main elements in the analysis}
\label{sec:main_elements}
Theorem~\ref{thm:main_result} critically depends on the general framework for compressed sensing in inverse problems developed in \cite{split1,split2}. In these works, the authors developed sufficiently flexible tools to study subsampled linear inverse problems with general vector-valued measurements and to derive precise recovery sample complexity estimates that relate three key quantities: the sparsity or compressibility of the signal, the number of measurements required for stable recovery, and the ill-posedness of the problem. In \cite{split2}, the result is based on three elements:
\begin{itemize}
    \item Stability bounds for the natural forward map $U$ associated with the inverse problem, referred to as \textit{quasi-diagonalization} bounds.
    \item A \textit{balancing property} that controls the stability in the discretization of the natural forward map $U$, leading to a truncated forward map $F$.
    \item \textit{Coherence bounds} for the measurement operators $F_t$ with respect to the analysis dictionary $(\phi_{j,n})_{(j,n)\in\Gamma}$ used in the reconstruction.
\end{itemize}
The stability properties required for the main result -- see Assumption~\ref{ass:stability} -- are inspired by the stability conditions for PAT studied in \cite{haltmeier-nguyen16}. In particular, as a consequence of the main results in \cite{haltmeier-nguyen16}, we are able to prove that such stability estimates hold when $\Omega$ is strictly convex -- see Proposition~\ref{lem:L2_stability}. The stability estimate in Lemma~\ref{lem:H12_stability} is then obtained from  Assumption~\ref{ass:stability} by complex interpolation; a proof of this lemma is provided in the Appendix.

Regarding the balancing property, in the case of PAT, it reduces to a bound of the following type -- see Proposition~\ref{prop:fond_prop1}:
\begin{equation}
\label{eq:projection_piecewise}
    \|(I - P_{\cF})u\|_{L^2(0,T;L^2(\Sigma))} \leq \theta \|u_0\|_{L^2(\bR^3)}\qquad u_0 \in \Span_{ \substack{(j,n) \\ j\leq j_0} } \{\phi_{j,n}\},
\end{equation}
where $P_{\cF}$ is the orthogonal projection of a function in $L^2(0,T;L^2(\Sigma))$ onto the space of (time-dependent) piecewise constant functions with respect to the partition $(E_i)_{i=1}^N$ and $\theta\in(0,1)$ is a sufficiently small constant. To obtain this estimate, we used ideas from \cite{hyvonen} concerning approximation of boundary measurements with electrodes in the context of the Calderón problem. Using a refined argument developed in \cite{rondi}, we obtain the bound:
\begin{align*}
    \|(I - P_{\cF})u\|_{L^2(0,T;L^2(\Sigma))} \lesssim \big(\max_i \diam_{\bR^3}(E_i)\big)^{1/2} \|u\|_{L^2(0,T;H^{1/2}(\Sigma))}.
\end{align*}
From here, additional ad hoc estimates are required, including leveraging the so-called Littlewood-Paley property of wavelets -- we refer the reader to \cite{Me} for further details.

Finally, the coherence bounds from Proposition~\ref{prop:fond_prop2} exploit, in a crucial way, the construction of tensorized wavelets and a local energy decay property of waves which, in odd dimensions, is a direct consequence of Huygens' principle. Using these elements, we can apply the result from \cite[Theorem 2.4]{split2} -- or, up to minor modifications, \cite[Theorem 3.11]{split1} -- to obtain Theorem~\ref{thm:main_result}.

To summarize, the main contribution of the present paper is to prove Theorem~\ref{thm:main_result}, together with the extension to more general sensing systems given by Theorem~\ref{thm:new-result}, by recasting them in the abstract framework of \cite[Theorem 2.4]{split2} through a careful analysis and refined estimates tailored to PAT. The three main elements in the analysis of Theorem~\ref{thm:main_result} are provided by the following results:
\begin{itemize}
    \item Studying the validity of stability bound given by Assumption~\ref{ass:stability} for certain domains $\Omega$ -- see Proposition~\ref{lem:L2_stability}, based on results from \cite{haltmeier-nguyen16}.
    \item Establishing a balancing property for the wave operator $U$ -- see Proposition~\ref{prop:fond_prop1} --, which in turn relies on an interpolation estimate given by Lemma~\ref{lem:H12_stability}.
    \item Quantifying precisely coherence bounds of PAT measurements with respect to wavelets -- see Proposition~\ref{prop:fond_prop2}.
\end{itemize}

\subsection{Structure of the paper}
The paper is organized as follows. In Sec.~\ref{sec:wavelets} we introduce the dictionary of wavelets used in the reconstruction and the related notion of sparsity. In Sec.~\ref{sec:freespace}, we discuss the model for the dynamics of the pressure wave and discuss some assumptions on the stability of the associated inverse wave problem. In Sec.~\ref{sec:detectors}, we introduce the discretization of measurements via the notion of detectors, and define the associated measurements. In Sec.~\ref{sec:main-result}, we state our main result (Theorem~\ref{thm:main_result}). In Sec.~\ref{sec:new-result}, we extend our main result to more general sensing systems considered in the literature (Theorem~\ref{thm:new-result}). In Sec.~\ref{sec:proofs}, we provide the proof of the results from Sec.~\ref{sec:setting}; the proof of the more technical Lemma~\ref{lem:H12_stability} is deferred to Appendix~\ref{appendix:stab_est}.

\section{Setting and main result}
\label{sec:setting}

\subsection{Wavelet dictionary and sparsity}
\label{sec:wavelets}
Our goal is to reconstruct a signal in $L^2(B_1)$, where $B_1$ is the unit ball in $\bR^3$, that is compressible with respect to a suitable wavelet dictionary. We briefly recall the construction of compactly supported separable wavelets in $\bR^3$ and define the space to which the input signal belongs. We refer the reader to \cite{Me,Ma} and \cite[Appendix A]{split2} for further details.

There exists an orthonormal basis of compactly supported elements $(\psi_n^3)_{n}\cup(\phi_{j,n,\varepsilon}^3)_{n,\varepsilon}$, where $n=(n_1,n_2,n_3)\in\bZ^3$, $j\in\bN$ and $\varepsilon=(\varepsilon_1,\varepsilon_2,\varepsilon_3)\in\{0,1\}^3\setminus\{(0,0,0)\}$, defined as follows:
\begin{equation}
\label{eq:wav_def}
\begin{split}
    \psi_n^3(x_1,x_2,x_3) &\coloneqq \psi(x_1-n_1)\psi(x_2-n_2)\psi(x_3-n_3),\\
    \phi_{j,n,\varepsilon}^3(x_1,x_2,x_3) &\coloneqq 2^{3/2 j} \phi^{\varepsilon_1}(2^j x_1 - n_1)\phi^{\varepsilon_2}(2^j x_2-n_2)\phi^{\varepsilon_3}(2^j x_3-n_3),
\end{split}
\end{equation}
where $\phi^0=\psi$ and $\phi^1=\phi$ are compactly supported functions in $L^2(\bR)$; moreover, $\phi$ has zero mean. The functions $\psi$ and $\phi$ are referred to as the scaling function and the mother wavelet of the dictionary, respectively.

In what follows, we will suppose that $\psi,\phi\in C^2$; such a choice is possible -- see, for instance, \cite[Proposition A.5]{split2}. In this case, the wavelet dictionary satisfies the following \textit{Littlewood-Paley property} \cite[Proposition A.7]{split2}: for each $|s|<2$, there exists constants $c,C>0$ such that
\begin{equation}
\label{eq:litt_paley}
    c \|u\|_{H^s(\bR^3)}^2 \leq \sum_{n} |\langle u,\psi_n^3 \rangle|^2 + \sum_{j,n,\varepsilon} 2^{2sj} |\langle u,\phi_{j,n,\varepsilon}^3 \rangle|^2 \leq C \|u\|_{H^s(\bR^3)}^2,\quad
    u\in H^s(\bR^3).
\end{equation}
We consider a dictionary defined as follows: up to relabeling, we denote $(\psi_n^3)_{n}\cup(\phi_{j,n,\varepsilon}^3)_{j,n,\varepsilon}$ as  $(\phi_{j,n})_{j,n}$ ($j=0$ corresponding to the low frequency elements $(\psi_n^3)_{n}$ and $j>0$ corresponding to the high frequencies elements $(\phi_{j,n,\varepsilon}^3)_{j,n,\varepsilon}$) and consider the index set $\Gamma$ defined as
\begin{align*}
    \Gamma \coloneqq \{(j,n)\colon\ \supp(\phi_{j,n})\cap B_1\neq \emptyset\}.
\end{align*}
We consider the following compact set:
\begin{equation}
\label{eq:Kdef}
    K \coloneqq
    \overline{\bigcup_{(j,n)\in\Gamma} \supp(\phi_{j,n}) }.
\end{equation}
We define the space
\begin{equation}
\label{eq:definitionH1}
    \cH_1 \coloneqq \overline{\Span}^{L^2}\{\phi_{j,n}\colon\ (j,n)\in\Gamma\}.
\end{equation}
We also define the following index sets
\begin{align*}
    \Lambda_j \coloneqq \{(j',n)\in\Gamma\colon\ j'=j\},\qquad
    \Lambda_{\leq j_0} \coloneqq \cup_{j\leq j_0} \Lambda_{j},
\end{align*}
the subspaces of $\cH_1$
\begin{align*}
    \cM_j \coloneqq \Span\{\phi_{j',n}\colon\ (j',n)\in\Lambda_j\},\quad
    \cM_{\leq j_0} \coloneqq \Span\{\phi_{j,n}\colon\ (j,n)\in\Lambda_{\leq j_0}\}
\end{align*}
and the corresponding orthogonal projections in $\cH_1$
\begin{align*}
    P_{j} \coloneqq P_{\cM_j},\qquad 
    P_{\leq j_0} \coloneqq P_{\cM_{\leq j_0}},\qquad P_{\leq j_0}^{\perp} \coloneqq I-P_{\leq j_0}.
\end{align*}
We also remark that $|\Lambda_j|\asymp 2^{3j}$ for each $j$, so that
\begin{align*}
    \log(|\Lambda_{\leq j_0}|) \asymp j_0.
\end{align*}
Finally, we define the analysis operator $\Phi\colon\cH_1\rightarrow\ell^2(\Gamma)$ as $\Phi{u}\coloneqq(\langle u,\phi_{j,n} \rangle)_{j,n}$ and the error of best $s$-sparse approximation with respect to the $\ell^1$-norm as follows:
\begin{align*}
    \sigma_s(u)_1 \coloneqq \inf\{ \|\Phi{u}-x\|_{1}\colon\ x\in\ell^2(\Gamma)\text{ is $s$-sparse} \},
\end{align*}
where by $s$-sparse we mean a sequence in $\ell^2(\Gamma)$ that has at most $s$ nonzero components.

\subsection{PDE model for PAT}
\label{sec:freespace}
Given a signal $u_0\in L^2(\Omega_0)$, where $\Omega_0\supset K$ and $K$ is as in \eqref{eq:Kdef}, we consider its evolution $u=u(x,t)$ being the weak solution to the following initial value problem for the wave equation -- see \cite{agranovsky2007} and the references therein:
\begin{equation}
\label{eq:waveIVP}
\begin{cases}
    u_{tt} - \Delta{u} = 0 & \text{in $\bR^3\times[0,+\infty)$},\\
    u(\cdot,0) = u_0 & \text{in $\bR^3$},\\
    u_t(\cdot,0) = 0 & \text{in $\bR^3$}.
\end{cases}
\end{equation}
For our purposes, the solution can be defined via Fourier multipliers as
\begin{align*}
    u(\cdot,t) = \cos\lc (-\Delta)^{1/2}t \rc u_0 = \cF^{-1}\lc \cos(2\pi |\cdot| t) \cF{u_0}\rc,
\end{align*}
where $\cF$ denotes the Fourier transform, defined by $\cF{u}(\xi)\coloneqq\int_{\bR^3} u(x)\exp(-2\pi i x\cdot\xi)\,\diff{x}$. The solution also admits an explicit integral representation, known as \textit{Kirchhoff's formula} \cite[Theorem 2, Section 2.4]{Ev}, valid for $u_0\in C^2_c(\bR^3)$:
\begin{equation}
\label{eq:kirchoff}
    u(x,t) =
    \frac{1}{4\pi t^2} \int_{\partial B_t(x)}
    \lc u_0(y)+\nabla{u_0}(y)\cdot(y-x) \rc\,\diff{\sigma(y)},\quad x\in\bR^3,\ t>0.
\end{equation}
In this case, $u\in C^1([0,T];H^1(\bR^3))$.

For $T>0$, we use $V\colon C^2_c(\bR^3)\rightarrow C^0([0,T];L^2(\bR^3))$ to denote the operator mapping an initial datum $u_0=u_0(x)$ to $u=u(x,t)$, $x\in\bR^3$ and $t\in[0,T]$, namely to the corresponding solution of the initial value problem \eqref{eq:waveIVP}.

In the sequel, we will be dealing with the restriction of the solution $u$ to $\Sigma=\partial\Omega$, where $\Omega$ is a smooth bounded domain such that $\Omega_0\Subset\Omega$. If $u_0\in C^2_c(\bR^3)$, then $u\in C^1([0,T];H^1(\bR^3))$ and therefore the operator $U\colon C_c^2(\bR^3)\rightarrow C^1([0,T];H^{1/2}(\Sigma))$, given by $$U u_0\coloneqq V u_0|_{\Sigma\times[0,T]},$$ is well-defined. We assume that the following stability bounds hold.
\begin{assumption}
\label{ass:stability}
There exists an open set $\Omega_1$ with $K\subset\Omega_0\Subset\Omega_1\Subset\Omega$, where $K$ is as in \eqref{eq:Kdef}, and there exists constants $c,C,T>0$ such that following stability estimates hold for every $u_0\in C_{c}^2(\Omega_1)$:
\begin{align*}
    c\|u_0\|_{L^2(\bR^3)} \leq &\|U u_0\|_{L^2(0,T;L^2(\Sigma))} \leq
    C\|u_0\|_{L^2(\bR^3)},\\
    &\|U u_0\|_{L^2(0,T;H^1(\Sigma))} \leq
    C\|u_0\|_{H^1(\bR^3)}.
\end{align*}
\end{assumption}
Here $H^1(\Sigma)$ denotes the Sobolev space of $L^2$ functions (with respect to the metric induced by $\bR^3$) on the manifold $\Sigma$ that are square summable together with their Lie derivatives. This assumption implies in particular that it is possible to extend $U$ by continuity and density to a bounded operator $U\colon L^2(\Omega_1)\rightarrow L^2(0,T;L^2(\Sigma))$ such that $U(H_0^1(\Omega_1))\subset L^2(0,T;H^1(\Sigma))$. The operator $U$ so defined is an isomorphism onto its image. Notice that the lower bound in the first inequality implies a restriction on $T$, which should be sufficiently large.

The following Proposition, proved in Sec.~\ref{sec:spherical}, shows that Assumption~\ref{ass:stability} holds when $\Omega$ is strictly convex.
\begin{proposition}
\label{lem:L2_stability}
Let $\Omega$ be a strictly convex bounded smooth domain, and let $\Omega_1\Subset \Omega$. There exists $T>0$ and constants $c,C>0$, depending only on $\Omega$ and on $d(\Omega_1,\Sigma)$, such that Assumption~\ref{ass:stability} is satisfied.
\end{proposition}

\subsection{Detectors and measurements}
\label{sec:detectors}
The measurements considered involve a finite number of \textit{detectors} on $\Sigma$, which we identify with closed subsets $E_i\subset\Sigma$ with $|E_i|>0$.

In order to derive our main result, we need to make some geometric assumptions on the detectors. In particular, we suppose that there exists a constant $C_{\mathrm{ecc}}>0$, fixed throughout the analysis, such that the following holds for every detector $E_i$:
\begin{equation}
\label{eq:ass_detectors}
    \diam_{\bR^3}(E_i) \leq C_{\mathrm{ecc}} |E_i|^{1/2},
\end{equation}
where the diameter is computed with respect to the Euclidean distance in $\bR^3$. This assumption is a sort of \textit{uniformly bounded eccentricity} condition on the detectors. To convince the reader that this assumption is scale-independent, we show in the next Proposition that a certain type of scaling (almost) preserves the eccentricity of detectors, provided they are sufficiently small. We will provide a proof of the Proposition in Sec.~\ref{sec:proof-ass-detectors}.
\begin{proposition}
\label{prop:ass_detectors}
Let $\mu>0$ and let $\Omega$ be a bounded smooth domain. There exists a constant $\delta=\delta(\mu,\Omega)>0$ such that the following holds. Let $\tilde{E}\subset B_{\delta}(0)$, where $B_{\delta}(0)$ is the ball of radius $\delta$ in $\bR^2$ centered in $0$, be such that
\begin{align*}
    \mathrm{diam}_{\bR^2}(\tilde{E}) \leq C_{\mathrm{ecc}} \cL^2(\tilde{E})^{1/2}
\end{align*}
for some constant $C_{\mathrm{ecc}}>0$, where $\cL^2$ is the Lebesgue measure in $\bR^2$. Let $p\in\Sigma$, $t\in(0,1]$ and
\begin{align*}
    E_t \coloneqq \exp_p( t\tilde{E} ),
\end{align*}
where $\exp_p\colon T_p\Sigma\rightarrow\Sigma$ is the exponential map with respect to the metric induced by $\bR^3$ on $\Sigma$.\footnote{We identify $\bR^2$ with the tangent space $T_p\Sigma$ via a linear isometry.}

Then
\begin{align*}
    \mathrm{diam}_{\bR^3}(E_t) \leq (1+\mu) C_{\mathrm{ecc}} |E_t|^{1/2}.
\end{align*}
\end{proposition}
In what follows, we fix a family $\cF=\{E_1,\dots,E_N\}$ of Lipschitz closed subsets of positive measure, referred to as \textit{detectors}, on $\Sigma$ that partition the surface in the following sense:
\begin{align*}
    \Sigma = \bigcup_i E_i,\qquad |E_i\cap E_j|=0,\ i\neq j.
\end{align*}
We select $m$ detectors $E_{i_1},\dots,E_{i_m}$ via a random sampling procedure. More precisely, let $\nu=(\nu_1,\dots,\nu_N)$ be the probability distribution defined on $\{1,\dots,N\}$ by
\begin{equation}
\label{eq:probability}
    \nu_i \coloneqq \frac{|E_i|}{|\Sigma|}.
\end{equation}
Notice that the probability of sampling a detector is proportional to its area; in other words, we choose larger detectors with higher probability. We then consider $m$ i.i.d.\ samples $i_1,\dots,i_m\sim\nu$ and the corresponding detectors $E_{i_1},\dots,E_{i_m}$.

The corresponding \textit{measurements} are then given by
\begin{equation}
\label{eq:measurements}
    y_k(t) \coloneqq \frac{1}{|E_{i_k}|}\int_{E_{i_k}} U u_0(y,t) \,\diff{\sigma}(y) + \varepsilon_k(t) \in L^2(0,T),
\end{equation}
where $\varepsilon_k\in L^2(0,T)$ are functions modeling noise on the data.

\subsection{Main result}
\label{sec:main-result}
We now state our main result.
\begin{theorem}
\label{thm:main_result}
Consider the setting introduced in sections \ref{sec:wavelets}, \ref{sec:freespace}, \ref{sec:detectors} and suppose that \eqref{eq:ass_detectors} is satisfied. Suppose that Assumption~\ref{ass:stability} is satisfied for some $K\subset\Omega_0\Subset\Omega_1\Subset\Omega$, where $K$ is as in \eqref{eq:Kdef}. Then there exist constants $\mu,C_0,C_1,C_2,C_3$, depending only on $T,\Omega_0,\Omega_1,\Omega$, on the wavelet basis, on the constants $c,C$ in Assumption~\ref{ass:stability}, and on $C_{\mathrm{ecc}}$, such that the following holds.

Consider $u_0\in L^2(B_1)$. Suppose that
\begin{equation}
\label{eq:detector_condition}
    \diam_{\bR^3}(E_i)\leq \mu 2^{-j_0},\qquad i=1,\dots,N,
\end{equation}
for some $j_0\in\bN$ and that $\|P_{\leq j_0}^{\perp} u_0\|_{L^2(\bR^3)}\leq r$. Consider $m$ i.i.d.\ samples\footnote{Sampling is performed
with replacement, and repeated indices are retained.} $i_1,\dots,i_m$ from the distribution $\nu$ on $\{1,\dots,N\}$ defined in \eqref{eq:probability}. Let $(\varepsilon_k)_{k=1}^m\in L^2(0,T)$ be such that $\|\varepsilon_k\|_{L^2(0,T)}\leq \beta$ for $k=1,\dots,m$ and let $(y_k)_{k=1}^m$ be given by \eqref{eq:measurements}.

Consider a solution $\widehat{u}$ to the following minimization problem:
\begin{equation}
\label{eq:minimization_problem}
    \min_{v\in\cM_{\leq j_0}} \|\Phi v\|_{\ell^1}\colon\qquad
    \frac{1}{m}\sum_{k=1}^m \bigg\|\frac{1}{|E_{i_k}|}\int_{E_{i_k}} U v\,\diff{\sigma}-y_k\bigg\|_{L^2(0,T)}^2 \leq \eta^2,
\end{equation}
where $\eta\coloneqq \beta+C_3\max\{1,(\min_i|E_i|)^{-1/2}\}\,r$. Let $\gamma\in(0,1)$. If
\begin{equation}
\label{eq:main_result_sample_complexity}
    m \geq C_0 s \max\{j_0\log^3{s},\log(1/\gamma)\}
\end{equation}
for some $s\geq 3$, then, with probability exceeding $1-\gamma$, the following holds:
\begin{equation}
\label{eq:recovery}
    \|u_0-\widehat{u}\|_{L^2} \leq
    C_1 \frac{\sigma_s(P_{\leq j_0} u_0)_1}{\sqrt{s}} + C_2\big( \beta+\max\{1,(\min_i|E_i|)^{-1/2}\}\,r \big).
\end{equation}
\end{theorem}

A few remarks are in order.
\begin{remark}
Condition \eqref{eq:detector_condition} on the size of the detectors guarantees stable recovery for signals in $\cM_{\leq j_0}$. This is a reasonable assumption: indeed, $\cM_{\leq j_0}$ is generated by wavelets at resolution at most $j_0$; those at the highest resolution $j_0$ are supported in a set of diameter $\asymp 2^{-j_0}$ -- see also the informal discussion in Section~\ref{sec:main_elements} for further insights on this point. It is plausible that the measurements should be given by averages on detectors of approximately the same size in order to get stable recovery, although we don't claim optimality for this scaling.
\end{remark}

\begin{remark}
\label{rem:detector_dependence}
The dependence of the recovery estimate \eqref{eq:recovery} on the truncation level is given by
\begin{align*}
    \max\{1,(\min_i|E_i|)^{-1/2}\}\,r.
\end{align*}
This quantity is optimized by choosing partitions with quasi-uniform detectors of maximal size, namely families $(E_i^{j_0})_{i=1}^{N_{j_0}}$ for each $j_0\in\bN$ such that
\begin{align*}
    c2^{-2j_0} \leq |E_i^{j_0}| \leq C 2^{-2j_0},\qquad \forall i=1,\dots,N_{j_0},
\end{align*}
where $c,C>0$ are constants independent of $j_0$. In this case, the error can be bounded by
\begin{align*}
    \max\{1,(\min_i |E_i|)^{-1/2}\}r \lesssim 2^{j_0} r.
\end{align*}
\end{remark}
The proof of Theorem~\ref{thm:main_result} is based on results from \cite{split1,split2}. As remarked after the statement of \cite[Theorem 3.5]{split1}, the abstract framework is suited for a uniform recovery result except for the truncation error -- see \cite[Proposition 5.8]{split1}. However, with some small effort, it would be possible to convert Theorem~\ref{thm:main_result} in a uniform recovery result; this is due to the fact that the inversion of the full forward map $U$ is well-posed due to Assumption~\ref{ass:stability}. Such a uniform extension, if established, would imply that, with high probability, the
same sampled detectors stably recover all $s$-sparse signals. We won't carry out this extension in the present work.

\begin{example}[Besov space prior]
\label{ex:besov}
Let $1\leq p\leq +\infty$ and $1\leq q<+\infty$, let
\begin{align*}
    t >
    \begin{cases}
        1+3\lc \frac{1}{p}-\frac{1}{2} \rc & \text{for $1\leq p<2$},\\
        1 & \text{otherwise}.
    \end{cases}
\end{align*}
We consider the Besov space $B_{p,q}^t(\bR^3)$ -- see \cite{bergh2012interpolation} for a definition. Suppose that $u_0\in B_{p,q}^t(\bR^3)$, $\supp(u_0)\subset B_1$ and $\|u_0\|_{B_{p,q}^t}\leq 1$. By the embedding theorem for Besov spaces \cite[Theorem~6.5.1]{bergh2012interpolation}, we have that, if $1\leq p\leq 2$, then $B_{p,q}^t(\bR^3)\subset B_{2,q}^{1+\varepsilon}(\bR^3)$ by our choice of $t$; if $p\geq 2$, then $u_0\in B_{2,q}^t(\bR^3)$ and $\|u_0\|_{B_{2,q}^t} \lesssim \|u_0\|_{B_{p,q}^t}$ due to the compact support of $u_0$. In both cases, we have that
\begin{align*}
    \|u_0\|_{B_{2,q}^r} \leq C
\end{align*}
for some $1<r<2$, $1\leq q<+\infty$ and $C>0$; we also suppose for convenience that $r\neq 3/2$. By the wavelet characterization of Besov spaces (see, for instance, \cite[Sec.~2.9, Prop.~4]{Me}), we have that
\begin{equation}
\label{eq:besov_characterization}
    c\|u_0\|_{B_{2,q}^r} \leq \lc \sum_j 2^{qrj} \|P_j u_0\|_{L^2}^q \rc^{1/q} \leq C\|u_0\|_{B_{2,q}^r}
\end{equation}
for some $C,c>0$.

We now estimate the error in Theorem~\ref{thm:main_result} for $u_0$, under the assumption that the detectors are of comparable size (see Remark~\ref{rem:detector_dependence}), so that the error is of the form
\begin{equation}
\label{eq:error_form}
    \|u_0-\widehat{u}\|_{L^2} \lesssim \beta + \frac{\sigma_s(P_{\leq j_0}u_0)_1}{\sqrt{s}} + 2^{j_0} \|P_{\leq j_0}^{\perp}u_0\|_{L^2}.
\end{equation}
We estimate the second term. As $|\Lambda_j|\asymp 2^{3j}$, we have that
\begin{align*}
    \sum_n |\langle u_0,\phi_{j,n} \rangle| \lesssim 2^{3/2 j} \|P_j u_0\|_{L^2}.
\end{align*}
We conclude that
\begin{align*}
    \frac{\sigma_s(P_{\leq j_0} u_0)_1}{\sqrt{s}} &\leq
    s^{-1/2} \sum_{j\leq j_0}\sum_n |\langle u_0,\phi_{j,n} \rangle| \\ &\lesssim
    s^{-1/2} \sum_{j\leq j_0} 2^{3/2 j} \|P_j u_0\|_{L^2} \\ &=
    s^{-1/2} \sum_{j\leq j_0} 2^{(3/2-r) j} 2^{rj}\|P_j u_0\|_{L^2} \\ &\lesssim
    s^{-1/2} 2^{(3/2-r)_+ j_0},
\end{align*}
where we have used \eqref{eq:besov_characterization} in the last step. Thus, choosing $s\asymp 2^{\alpha j_0}$, we get that
\begin{align*}
    \frac{\sigma_s(P_{\leq j_0} u_0)_1}{\sqrt{s}} \lesssim 2^{\lc (3/2-r)_+-\alpha/2 \rc j_0}.
\end{align*}
Notice that, as $(3/2-r)_+-\alpha/2 \leq (1-\alpha)/2$, choosing $\alpha>1$ implies that the RHS converges to $0$ as $j_0\rightarrow+\infty$.

We now estimate the third term in \eqref{eq:error_form}. We have that
\begin{align*}
    2^{j_0} \|P_{\leq j_0}^{\perp}u_0\|_{L^2} &=
    2^{j_0} \lc \sum_{j>j_0} \|P_j{u_0}\|_{L^2}^2\rc^{1/2} \lesssim
    2^{(1-r) j_0},
\end{align*}
which converges to $0$ in $j_0$ as $r>1$. Thus we have that
\begin{equation}
\label{eq:error_form2}
    \|u_0-\widehat{u}\|_{L^2} \lesssim \beta + 2^{\lc (3/2-r)_+-\alpha/2 \rc j_0} + 2^{(1-r) j_0},
\end{equation}
where $s\asymp 2^{\alpha j_0}$. It is then possible to choose a suitable $j_0=j_0(\beta)$ so that the RHS is proportional to the noise level $\beta$. Moreover, by \eqref{eq:main_result_sample_complexity}, we have that a sufficient number of detectors $m$ to obtain \eqref{eq:error_form2} with high probability is given by
\begin{align*}
    m \gtrsim j_0 s \cdot (\text{log factors}) \asymp 2^{\alpha j_0}\cdot (\text{log factors}).
\end{align*}
Thus, if $m\asymp 2^{\alpha j_0}\cdot\text{(log factors)}$ and $\alpha<2$, then the ratio between $m$ and the total number of detectors $N\asymp 2^{2j_0}$ converges to $0$ in $j_0$.

We stress that this example does not contradict the instability result from \cite[Theorem 3.1]{nguyen2011}; indeed, the latter proves Hölder instability for the recovery of a signal in a ball of some Sobolev space at infinite resolution for a fixed detector configuration. On the other hand, in the current example, the resolution $j_0$ must be chosen according to the noise level $\beta$, and the randomly sampled detector configuration is thus resolution-dependent.
\end{example}

\begin{exampleoutlook}[Vascular imaging]
Another example, inspired by vascular imaging, would be to consider the recovery of signals that can be approximated by functions of the form $u_0=\mu\ast\eta_{\delta}$, where $\mu$ is a finite measure supported on a compact smooth curve $\gamma\subset B_1$ and $\eta_{\delta}$ is a compactly supported mollifier, with $\delta>0$ sufficiently small. In this case, for scales $2^{-j_0}$ comparable to $\delta$, the number of wavelets whose support intersects $\supp(u_0)$ non-trivially scales like $s\asymp \ell(\gamma)\delta^{-1}$, where $\ell(\gamma)$ is the length of the curve $\gamma$. If we ignore the truncation error, the sample complexity in Theorem~\ref{thm:main_result} would give a number of detectors
\begin{align*}
    m\asymp j_0 s\cdot\text{(log factors)} \asymp \ell(\gamma)\delta^{-1}\cdot\text{(log factors)},
\end{align*}
which is much smaller than the total number $N\asymp 2^{2j_0}\asymp \delta^{-2}$.

The main issue here is that the truncation error $2^{j_0}\|P_{\leq j_0}^{\perp} u_0\|_{L^2}$ cannot be shown to vanish as $\delta\rightarrow 0^+$, and therefore such estimates would be unstable with respect to the vessel width. Indeed, for sufficiently regular wavelets, the truncation error can be controlled by
\begin{align*}
    2^{j_0}\|P_{\leq j_0}^{\perp} u_0\|_{L^2} \lesssim 2^{-j_0} \|u_0\|_{H^2} \lesssim 2^{-j_0} \delta^{-7/2} \asymp \delta^{-5/2},
\end{align*}
which does not converge as $\delta\rightarrow 0^+$. The bound on the $H^2$ norm of $u_0$ follows from
\begin{align*}
    \|\partial_{ij}^2(\mu\ast \eta_{\delta})\|_{L^2} &=
    \|\mu\ast (\partial^2_{ij}\eta_{\delta})\|_{L^2} \leq
    |\mu|(\Omega) \|\partial^2_{ij}\eta_{\delta}\|_{L^2} \lesssim |\mu|(\Omega) \delta^{-7/2}.
\end{align*}
A solution would be to extend Theorem~\ref{thm:main_result} to negative Sobolev norms $H^{-t}$ replacing the $L^2$ norm, and to show that $2^{j_0}\|P_{\leq j_0}^{\perp} u_0\|_{H^{-t}}$ converges to $0$ as $\delta\rightarrow 0^+$. We leave such extensions to future work.
\end{exampleoutlook}

\subsection{Toward more general sensing patterns}
\label{sec:new-result}
We now investigate possible ways to extend our result to other sensing systems that have been considered in the literature.

Our main result -- Theorem~\ref{thm:main_result} -- assumes that each measurement is given by the averaged pressure wave on a single detector. In practice, it is possible to leverage the principles behind the so-called \textit{single-pixel camera} systems \cite{duarte2008,goda2009} to achieve better compression rates -- see, for instance, \cite{huynh2016,betcke2017}. These systems simplify the task of obtaining highly resolved spatial measurements by converting it into sequential recording in time using a fixed array of detectors. For these types of sensing patterns, multiple detectors activate at each measurement; the measurement is then given by a weighted sum of the averaged pressure wave on each detector.

We now show how it is possible to extend our framework to this setting. We assume, for simplicity, that there exists $C_u\geq 1$, fixed throughout the analysis, such that the detectors $\cF=(E_i)_{i=1}^N$ satisfy the following \textit{quasi-uniformity} assumption:
\begin{equation}
\label{eq:quasi-unif}
    \frac{1}{C_u} |E_i| \leq |E_j| \leq C_u |E_i|,\qquad i,j=1,\dots,N.
\end{equation}
This implies in particular that
\begin{align*}
     \frac{|\Sigma|}{N C_u} \leq |E_i| \leq \frac{|\Sigma| C_u}{N}.
\end{align*}
We suppose that the pressure wave measurements on the detectors is corrupted by noise $\varepsilon_k\in L^2(0,T)$ for $k=1,\dots,m$.

For convenience of notation, we define the following averaging operators:
\begin{align*}
    M_i\colon L^2(0,T;L^2(\Sigma))\rightarrow L^2(0,T),\qquad
    M_i{v}(t) \coloneqq 
    \frac{1}{|E_i|} \int_{E_i} v(y,t)\,\diff{\sigma}(y)
\end{align*}
and the associated operator $M$ defined by
\begin{align*}
    Mv(t) \coloneqq \lc M_1 v(t),\dots,M_N v(t) \rc\in \big(L^2(0,T)\big)^N.
\end{align*}
We fix a unitary matrix $A\in\bC^{N\times N}$.\footnote{A more general result would hold for invertible matrices; in this case, the constants appearing in the estimates would also depend on $\max\{\|A\|,\|A^{-1}\|\}$, where $\|\cdot\|$ denotes the operator norm.} Let $i_1,\dots,i_m$ be $m$ i.i.d.\ samples from the uniform distribution on $\{1,\dots,N\}$. The measurements in the new setting are given by
\begin{equation}
\label{eq:new-measurements}
    y_k(t) \coloneqq P_{i_k}AMU u_0 + \varepsilon_k,\qquad k=1,\dots,m,
\end{equation}
where $P_i\colon \big(L^2(0,T)\big)^N \rightarrow L^2(0,T)$ is the projection on the $i$-th component. Notice that, for $A=I$, the measurements reduce to the ones defined in \eqref{eq:measurements} with $\nu_i \asymp 1/N$, thanks to \eqref{eq:quasi-unif}.

We now state the new result, whose proof is provided in Sec.~\ref{sec:proof-new-result}.
\begin{theorem}
\label{thm:new-result}
Consider the setting introduced in sections \ref{sec:wavelets}, \ref{sec:freespace}, \ref{sec:detectors} and suppose that \eqref{eq:ass_detectors} and \eqref{eq:quasi-unif} are satisfied. Suppose that Assumption~\ref{ass:stability} is satisfied for some $K\subset\Omega_0\Subset\Omega_1\Subset\Omega$, where $K$ is as in \eqref{eq:Kdef}. There exist constants $\mu,C_0,C_1,C_2,C_3$, depending only depending only on $T,\Omega_0,\Omega_1,\Omega$, on the wavelet basis, on the constants $c,C$ in Assumption~\ref{ass:stability}, on $C_{\mathrm{ecc}}$ and on $C_u$, such that the following holds.

Consider $u_0\in L^2(B_1)$. Suppose that \eqref{eq:detector_condition} is satisfied for some $j_0\in\bN$ and that $\|P_{\leq j_0}^{\perp} u_0\|_{L^2(\bR^3)}\leq r$. Consider $m$ i.i.d.\ samples\footnote{As in Theorem~\ref{thm:main_result}, sampling is performed
with replacement, and repeated indices are retained.} $i_1,\dots,i_m$ from the uniform distribution on $\{1,\dots,N\}$. Let $(\varepsilon_k)_{k=1}^m\in L^2(0,T)$ be such that $\|\varepsilon_k\|_{L^2(0,T)}\leq \beta$ for $k=1,\dots,m$ and let $(y_k)_{k=1}^m$ be given by \eqref{eq:new-measurements}. Let
\begin{equation}
\label{eq:new-coherence}
    B \coloneqq \max\Big\{1,\max_{ \substack{i=1,\dots,N \\ (j,n)\in\Lambda_{\leq j_0}} } \| P_i A M U \phi_{j,n} \|_{L^2(0,T)}\Big\}.
\end{equation}
Consider a solution $\widehat{u}$ to the following minimization problem:
\begin{equation}
    \min_{v\in\cM_{\leq j_0}} \|\Phi v\|_{\ell^1}\colon\qquad
    \frac{1}{m}\sum_{k=1}^m \|P_{i_k} A M U v-y_k \|_{L^2(0,T)}^2 \leq \eta^2,
\end{equation}
where $\eta\coloneqq \beta+C_3\sqrt{N} r$. Let $\gamma\in(0,1)$. If
\begin{equation}
\label{eq:new-sample-complexity}
    m \geq C_0 \tau \max\{j_0\log^3{\tau},\log(1/\gamma)\},\qquad \tau\coloneqq B^2 s
\end{equation}
for some $s\geq 3$, then, with probability exceeding $1-\gamma$, the following holds:
\begin{equation}
    \|u_0-\widehat{u}\|_{L^2} \leq
    C_1 \frac{\sigma_s(P_{\leq j_0} u_0)_1}{\sqrt{s}} + C_2\big( \beta+\sqrt{N} r \big).
\end{equation}
\end{theorem}
We remark that the main drawback of applying Theorem~\ref{thm:new-result} to general sensing systems considered in the literature is of a technical nature: indeed, its proof relies on results from \cite{split1,split2} that only applies to wavelet orthonormal bases $(\phi_{j,n})_{j,n}$. This prevents us to apply our estimates to settings where such dictionaries are replaced by certain frames $(\phi_{j,n,l})_{j,n,l}$ that exhibit higher sensitivity to directionality. Examples of such dictionaries include \textit{curvelets} and \textit{shearlets} \cite{CD,GL,kuty_cartoon}.

Nevertheless, we argue that it is instructive to informally discuss the scaling of $B$ in \eqref{eq:new-coherence} in the cases where $(\phi_{j,n})_{j,n}$ is given by wavelets and by curvelets/shearlets. Indeed, the parameter $B$ plays a crucial role in the sample complexity from Theorem~\ref{thm:new-result}, and is closely related to the classical concept of \textit{coherence} in compressed sensing.

As previously noted, when $ A = I $, the setting of Theorem~\ref{thm:new-result} reduces to that of Theorem~\ref{thm:main_result} with $\nu_i\equiv 1/N$.\footnote{Indeed, although in Theorem~\ref{thm:main_result} the probability was chosen as $\nu_i=|E_i|/|\Sigma|$, the same argument would work for a probability mass for which $\nu_i\asymp|E_i|/|\Sigma|$.} By Proposition~\ref{prop:fond_prop2}, we obtain that, in this case, $ B $ can be bounded by a constant independent of $ j_0 $.\footnote{Indeed, notice that, with reference to the notation of Sec.~\ref{sec:proof-main-result}, the following identity holds: $$P_i M U = |E_i|^{-1/2} F_i .$$} This bound arises from the geometry of tensorized wavelets. Indeed, high-frequency wavelets exhibit low sensitivity to directionality; as a result, their energy propagates across a wide range of directions. Consequently, the corresponding pressure wave can be detected by localized measurements on a sparse array of detectors.  

In the case of curvelets and shearlets, instead, the essential support in frequency of the elements $(\phi_{j,n,l})_{j,n,l}$ becomes localized along cones whose aperture decreases as the resolution $j$ increases. Due to the propagation of singularities for the wave equation, the energy of the pressure wave associated with high-frequency elements in the dictionary thus becomes localized to a small number of detectors aligned with specific directions. In other words, we expect the time-dependent vector $ MU\phi_{j,n,l} $ to be approximately sparse. For this reason, it is preferable to adopt a different sensing pattern that promotes incoherence in the measurements. One possible choice is to use a binary matrix $ A $ with components in $ \{\pm 1/\sqrt{N}\} $; examples include, for instance, scrambled Hadamard matrices. For further information, we refer the reader to \cite{foucartrauhut} or to the appendix in \cite{haltmeier2016}. To heuristically estimate the coherence parameter $B$ in this case, we notice that, if $u=U\phi_{j,n}$ and $A=(a_{ik})_{i,k=1}^N$ with $|a_{ik}|=N^{-1/2}$, then
\begin{align*}
    \|P_i A M U\phi_{j,n}\|_{L^2(0,T)} = \bigg\|\sum_{k} a_{ik} M_k U\phi_{j,n}\bigg\|_{L^2(0,T)} \lesssim \sqrt{\frac{k_0}{N}} \|MU\phi_{j,n}\|_{\lc L^2(0,T)\rc^N},
\end{align*}
where $k_0\in\bN$ is the sparsity of $MU\phi_{j,n}(t)=(M_i U\phi_{j,n})_i$. We conclude that
\begin{align*}
    \|P_i A M U\phi_{j,n}\|_{L^2(0,T)} &\lesssim 
    \sqrt{\frac{k_0}{N}} \|MU\phi_{j,n}\|_{\lc L^2(0,T)\rc^N} \lesssim \sqrt{k_0} \|\phi_{j,n}\|_{L^2(\bR^3)} = \sqrt{k_0}.
\end{align*}
Thus, in this case, $B$ can be bounded by a constant which only depends on the detector sparsity $k_0$. Similar estimates can be provided in the case where the observation on the detectors is only approximately sparse, which, according to the previous discussion on propagation of singularities, is the situation expected for curvelets and shearlets.

\section{Proofs}
\label{sec:proofs}
In this section, we provide the proof of the statements from Sec.~\ref{sec:setting}.

\subsection{Proof of Proposition~\ref{lem:L2_stability}}
\label{sec:spherical}
The upper bound for the $L^2$ norm follows directly from \cite[Theorem~3.1]{haltmeier-nguyen16} in the case $i=0$. We now prove the upper bound for the $H^1$ norm. Let $u=Vu_0$; then, differentiating the wave equation (recall that the wave speed is constant), we conclude that $\partial_{x_i}{u}=V(\partial_{x_i}u_0)$. Applying the $L^2$ stability estimate to $\partial_{x_i}u_0$ leads to
\begin{align*}
    \|\nabla_{\Sigma}u\|_{L^2(0,T;L^2(\Sigma))} \lesssim
    \|\nabla{u}\|_{L^2(0,T;L^2(\Sigma))} \lesssim
    \|\nabla{u_0}\|_{L^2(\bR^3)} \lesssim \|u_0\|_{H^1(\bR^3)}.
\end{align*}
Finally, we prove the lower bound on the $L^2$ norm. In the case of constant speed, the visibility condition from \cite[Theorem~3.4]{haltmeier-nguyen16} is clearly satisfied for a sufficiently large $T>0$; applying the Theorem in the case $i=0$ implies the claimed lower bound. This concludes the proof.

\subsection{Proof of Proposition~\ref{prop:ass_detectors}}
\label{sec:proof-ass-detectors}
We have that
\begin{align*}
    d\exp_p 0 = I.
\end{align*}
Therefore, choosing a sufficiently small $\delta>0$, we have that
\begin{align*}
    (1-\varepsilon)|h| \leq |d\exp_p(v)[h]| \leq (1+\varepsilon)|h|,\quad v\in B_{\delta}(0),\ h\in T_p\Sigma,
\end{align*}
where $\varepsilon=\varepsilon(\mu)>0$ will be chosen later. Such $\delta>0$ can be chosen to be independent of $p$ due to the compactness of $\Sigma$. We then conclude that
\begin{align*}
    |\exp_p(tv)-\exp_p(tw)| \leq (1+\varepsilon)t|v-w|,\quad v,w\in \tilde{E},
\end{align*}
and therefore
\begin{align*}
    \mathrm{diam}_{\bR^3}(E_t) \leq (1+\varepsilon)t\, \mathrm{diam}_{\bR^2}(\tilde{E}) \leq
    (1+\varepsilon)t\,C_{ecc}\cL^2(\tilde{E})^{1/2}.
\end{align*}
We can additionally suppose that $\delta>0$ is smaller than the injectivity radius of $\Sigma$, so that $\exp_p$ is injective on $B_{\delta}(0)$. The area formula then implies that
\begin{align*}
    |E_t| = t^2\int_{\tilde{E}} \sqrt{\det\lc (d\exp_p)^{\top} (d\exp_p) \rc(tv)}\,\diff{v} \geq (1-\varepsilon)^2 t^2 \cL^2(\tilde{E}).
\end{align*}
This implies that
\begin{align*}
    \mathrm{diam}_{\bR^3}(E_t) \leq \frac{1+\varepsilon}{1-\varepsilon} C_{ecc}|E_t|^{1/2}.
\end{align*}
Choosing $\varepsilon>0$ such that
\begin{align*}
    \frac{1+\varepsilon}{1-\varepsilon} \leq 1+\mu,
\end{align*}
concludes the proof.

\subsection{Proof of the main result}
\label{sec:proof-main-result}
Our proof of Theorem~\ref{thm:main_result} is based on \cite[Theorem 2.4]{split2}; we briefly recall here the setup for the convenience of the reader.

The measurement procedure is modeled via a family of bounded measurement operators $(F_t)_{t\in\cD}\colon\cH_1\rightarrow\cH_2$ between Hilbert spaces $\cH_1,\cH_2$. In our case, $\cH_1$ is given by \eqref{eq:definitionH1}, while it is convenient to write $F_t$ as $F_i $ for $i=1,\dots,N$, where
\begin{equation}
\label{eq:measurement_operators_definition}
    F_i  u_0(t) \coloneqq |E_i|^{1/2}M_i U u_0 = \frac{1}{|E_i|^{1/2}}\int_{E_i} U u_0(y,t)\,\diff{\sigma(y)}.
\end{equation}
The space $\cD$ that parametrizes the measurements corresponds then to $\{1,\dots,N\}$ and $\cH_2$ is $L^2(0,T)$. The sampling procedure is modeled via a probability distribution $\nu$ on $\cD$; in our case, it is given by the probability $\nu=(\nu_1,\dots,\nu_N)$ on ${1,\dots,N}$ defined by \eqref{eq:probability}.\footnote{In \cite{split2}, the probability distribution was defined by a probability density function $f_{\nu}(t)$ with respect to another measure $\mu$ on $\cD$, which in our case is the counting measure on $\{1,\dots,N\}$. Then the correspondence between the notation of the present paper and the notation of \cite{split2} is given by $\nu_i=f_{\nu}(i)$.} The forward map $F\colon\cH_1\rightarrow \lc L^2(0,T)\rc^N$ modeling the collection of all possible measurements is given by
\begin{align*}
    F u_0 \coloneqq (F_i u_0)_{i=1}^N \in (L^2(0,T))^N.
\end{align*}
We now compute the norm of $F u_0$:
\begin{equation}
\label{eq:forward_norm}
\begin{split}
    \|F u_0\|_{\lc L^2(0,T)\rc^N}^2 &=
    \sum_{i=1}^N \|F_i  u_0\|_{L^2(0,T)}^2 \\ &=
    \sum_{i=1}^N \int_{\Sigma}
     \|F_i u_0\|_{L^2(0,T)}^2\frac{\mathbbm{1}_{E_i}(y)}{|E_i|}\,\diff{\sigma(y)} \\ &=
    \int_0^T \int_{\Sigma}
    \bigg| \sum_{i=1}^N F_i u_0(t)\frac{\mathbbm{1}_{E_i}(y)}{|E_i|^{1/2}} \bigg|^2\,\diff{\sigma(y)}\diff{t},
\end{split}
\end{equation}
where we carried the summation inside the square in the last equality using the fact that $|E_i\cap E_j|=0$ for $i\neq j$. The previous formula shows that $\|F u_0\|_{\lc L^2(0,T)\rc^N}$ is given by $\|P_{\cF}U u_0\|_{L^2(0,T;L^2(\Sigma))}$, $P_{\cF}$ being the orthogonal projection on $L^2(0,T;\cP_{\cF})$, where $\cP_{\cF}$ is the space of piecewise-constant functions in the spatial variable with respect to $\cF=\{E_1,\dots,E_N\}$. Explicitly, we have that
\begin{equation}
\label{eq:projection}
    P_{\cF}U u_0(t,y) = \sum_{i=1}^N F_i u_0(t) \frac{\mathbbm{1}_{E_i}(y)}{|E_i|^{1/2}}.
\end{equation}
This implies that, up to an isomorphism, the forward map $F$ can be identified with $P_{\cF}U$, where $P_{\cF}$ is a projection operator on $L^2(0,T;L^2(\Sigma))$. The map $U$ is referred to as the \textit{natural forward map} in \cite[Section 2.4]{split2}.

The assumptions to be verified in order to apply the nonuniform bound case in \cite[Theorem 2.4]{split2} are summarized in what follows.
\begin{enumerate}
    \item[$(a)$] There exists constants $c,C>0$ such that the following stability property holds:
    \begin{equation}
    \label{eq:weak-quasi-diag}
        \begin{split}
            &\|F u_0\|_{(L^2(0,T))^N}^2 \leq C \|u_0\|_{L^2(K)}^2,\qquad u_0\in\cH_1,\\
            &\frac{c}{2}\|u_0\|_{L^2(K)}^2 \leq \|F u_0\|_{(L^2(0,T))^N}^2,\qquad u_0\in\cM_{\leq j_0}.
        \end{split}
    \end{equation}
    \item[$(b)$] There exists a constant $B\geq 1$ such that the following \textit{coherence bounds} holds:
    \begin{equation}
    \label{eq:coherence_bounds}
        \|F_i \phi_{j,n}\|_{L^2(0,T)} \leq B\sqrt{\nu_i},\qquad i\in\{1,\dots,N\},\ (j,n)\in\Lambda_{\leq j_0}.
    \end{equation}
\end{enumerate}
We now state a version of the abstract result \cite[Theorem 2.4]{split2} in the nonuniform bound case adapted to the present context.
\begin{theorem}
\label{thm:split2}
Consider the previous setting. Fix $j_0\in\bN$ and $\gamma\in(0,1)$. Consider:
\begin{itemize}
    \item a signal $u_0\in\cH_1$;
    \item i.i.d.\ samples $i_1,\dots,i_m\in\{1,\dots,N\}$ from $\nu=(\nu_1,\dots,\nu_N)$;
    \item measurements $\tilde{y}_k=F_{i_k}u_0+\tilde{\varepsilon}_k$, $k=1,\dots,m$, for some $\tilde{\varepsilon}_k\in\cH_2$;
    \item a sparsity level $s\geq 3$.
\end{itemize}
Moreover, suppose that the following assumptions are satisfied:
\begin{itemize}
    \item the noise $\tilde{\varepsilon}_k$ satisfies
    \begin{equation}
    \label{eq:noise-ass}
        \max_{k=1,\dots,m} \|\nu_{i_k}^{-1/2} \tilde{\varepsilon}_k\|_{\cH_2} \leq \tilde{\beta}
    \end{equation}
    for some $\tilde{\beta}\geq 0$;
    \item the truncation error $u_R\coloneqq P_{\leq j_0}^{\perp}u_0$ satisfies
    \begin{equation}
    \label{eq:truncation-error}
        \max_{i=1,\dots,N} \nu_i^{-1/2} \|F_i u_R\|_{\cH_2} \leq \tilde{r},\qquad
        \|u_R\|_{\cH_1} \leq \tilde{r}
    \end{equation}
    for some $\tilde{r}\geq 0$.
\end{itemize}
There exists constants $C_0,C_1,C_2,C_3'>0$, depending only on the constants $c,C>0$ in \eqref{eq:weak-quasi-diag}, such that the following holds. Let
\begin{align*}
    m \geq C_0 \tau \max\{\log^3{\tau}\log{|\Lambda_{\leq j_0}|},\log(1/\gamma)\},\qquad \tau=B^2 s.
\end{align*}
Let $\widehat{u}$ be a solution of
\begin{equation}
\label{eq:minimization-split2}
    \min_{u\in\cM_{\leq j_0}} \|\Phi{u}\|_{\ell^1}\colon\qquad 
    \frac{1}{m}\sum_{k=1}^m \nu_{i_k}^{-1} \|F_{i_k}u-\tilde{y}_k\|_{\cH_2}^2 \leq (\tilde{\beta}+C_3'\tilde{r})^2.
\end{equation}
Then, with probability greater than $1-\gamma$, the following recovery estimate holds:
\begin{align*}
    \|u_0-\widehat{u}\|_{\cH_1} \leq C_1 \frac{\sigma_s(P_{\leq j_0} u_0)}{\sqrt{s}} + C_2(\tilde{\beta}+\tilde{r}).
\end{align*}
\end{theorem}
As noticed in \cite[Remark 2.5]{split2}, condition \eqref{eq:weak-quasi-diag}, which is slightly weaker than the quasi-diagonalization condition in the statement of \cite[Theorem 2.4]{split2}, is sufficient for the theorem to hold.

Using \cite[Proposition 2.11]{split2}, we can reduce condition $(a)$ to the following two conditions.
\begin{itemize}
    \item[$(a1)$] There exists constants $c,C>0$ such that the following stability property holds:
    \begin{equation}
    \label{eq:quasi-diagonalization}
        c\|u_0\|_{L^2(K)}^2 \leq 
        \|U u_0\|_{L^2(0,T;L^2(\Sigma))}^2 \leq C\|u_0\|_{L^2(K)}^2,\quad u_0\in\cH_1,
    \end{equation}
    where we recall that $\cH_1$ was defined in \eqref{eq:definitionH1}.
    \item[$(a2)$] The projection $\cP_{\cF}$ satisfies the following \textit{balancing property}:
    \begin{align*}
        \|(I-\cP_{\cF})U u_0\|_{L^2(0,T;L^2(\Sigma))}^2 \leq \theta^2 \|u_0\|_{L^2(K)}^2,\qquad u_0\in \cM_{\leq j_0}
    \end{align*}
    for $\theta=\sqrt{c/2}$, where $c$ is the constant in \eqref{eq:quasi-diagonalization}.
\end{itemize}

Therefore, the main part of the proof of Theorem~\ref{thm:main_result} is to prove conditions $(a1)$, $(a2)$ and $(b)$. Condition $(a1)$ is satisfied by Assumption~\ref{ass:stability}; the results corresponding to assumptions $(a2)$ and $(b)$ are stated in Proposition~\ref{prop:fond_prop1} and \ref{prop:fond_prop2}, respectively, and are proved below.
\begin{proposition}
\label{prop:fond_prop1}
Consider the setting introduced in sections \ref{sec:wavelets}, \ref{sec:freespace}, \ref{sec:detectors} and suppose that \eqref{eq:ass_detectors} is satisfied. Suppose that Assumption~\ref{ass:stability} is satisfied for some $K\subset\Omega_0\Subset\Omega_1\Subset\Omega$, where $K$ is as in \eqref{eq:Kdef}. Let $\theta>0$. Then there exists $\mu>0$, depending only on $\theta$, on the wavelet basis, on $\Omega_0,\Omega_1,\Omega$, on the constants $c,C$ in Assumption~\ref{ass:stability}, and on $C_{\mathrm{ecc}}$, such that, if the following condition holds
\begin{equation}
\label{eq:fond_prop1_ass_det}
    \diam_{\bR^3}(E_i) \leq \mu 2^{-j_0},\qquad i=1,\dots,N,
\end{equation}
for some $j_0\in\bN$, then the following inequality holds:
\begin{equation}
\label{eq:forw_stab}
    \|(I-\cP_{\cF}) U u_0\|_{L^2(0,T;L^2(\Sigma))} \leq \theta\|u_0\|_{L^2(K)},\qquad u_0\in\cM_{\leq j_0}.
\end{equation}
\end{proposition}
In the proof of Proposition~\ref{prop:fond_prop1}, we need stability estimates with respect to the $H^{1/2}$ norm, which we state in the next Lemma. The proof of this result is provided in Appendix~\ref{appendix:stab_est}.
\begin{lemma}
\label{lem:H12_stability}
Consider the setting introduced in sections \ref{sec:wavelets}, \ref{sec:freespace}, \ref{sec:detectors}. Suppose that Assumption~\ref{ass:stability} is satisfied for some $K\subset\Omega_0\Subset\Omega_1\Subset\Omega$, where $K$ is as in \eqref{eq:Kdef}. Then there a constant $C_1>0$, depending only on $\Omega_0,\Omega_1,\Omega$, such that the following holds for $u_0\in C^2_c(\Omega_0)$:
\begin{align*}
    &\|U u_0\|_{L^2(0,T;H^{1/2}(\Sigma))} \leq C_1\|u_0\|_{H^{1/2}(\bR^3)}.
\end{align*}
\end{lemma}
Here $H^{1/2}(\Sigma)$ denotes the corresponding Sobolev–Slobodeckij space -- we refer the reader to \cite{valdinoci}, for instance, for a definition.
\begin{proof}[Proof of Proposition~\ref{prop:fond_prop1}]
Recall formula \eqref{eq:projection} for $P_{\cF}U$. We now prove that, if \eqref{eq:fond_prop1_ass_det} holds for some $\mu>0$, then the following holds for $u_0\in\cM_{\leq j_0}$:
\begin{equation}
\label{eq:fond_est_prop1}
    \|(I-P_{\cF})U u_0\|_{L^2(0,T;L^2(\Sigma))}^2 \leq C'\mu \|u_0\|_{L^2(K)}^2,
\end{equation}
for some constant $C'>0$ which depends only on the wavelet basis, on $\Omega_0,\Omega_1,\Omega$ and on $C_{\mathrm{ecc}}$. Choosing $\mu\coloneqq \frac{\theta^2}{C'}$ will then imply \eqref{eq:forw_stab}, which concludes the proof.

We proceed with the proof of \eqref{eq:fond_est_prop1}. We have that
\begin{align*}
    &\|U u_0 - P_{\cF}U u_0\|_{L^2(0,T;L^2(\Sigma))}^2 \\ &\hspace{5mm}=
    \sum_{i=1}^N \int_0^T \int_{E_i} \bigg| 
U u_0(t,y)-\frac{1}{|E_i|}\int_{E_i} U u_0(t,x)\,\diff{\sigma(x)} \bigg|^2\,\diff{\sigma(y)}\diff{t}.
\end{align*}
We will estimate each term separately. We have that
\begin{align*}
    &\int_0^T \int_{E_i} \bigg| 
    U u_0(t,y)-\frac{1}{|E_i|}\int_{E_i} U u_0(t,x)\,\diff{\sigma(x)} \bigg|^2\,\diff{\sigma(y)}\diff{t} \\ &\hspace{5mm}=
    \int_0^T \int_{E_i} \bigg| \int_{E_i} 
    \lc U u_0(t,y)-U u_0(t,x) \rc\,\frac{\diff{\sigma(x)}}{|E_i|} \bigg|^2\,\diff{\sigma(y)}\diff{t} \\ &\hspace{5mm}\leq
    \frac{1}{|E_i|}\int_0^T \int_{E_i\times E_i}
    |U u_0(t,y)-U u_0(t,x)|^2\,\diff{(\sigma\otimes\sigma)(x,y)}\diff{t} \\ &\hspace{5mm}=
    \frac{1}{|E_i|}\int_0^T \int_{E_i\times E_i} |x-y|^3
    \frac{|U u_0(t,y)-U u_0(t,x)|^2}{|x-y|^3}\,\diff{(\sigma\otimes\sigma)(x,y)}\diff{t} \\ &\hspace{5mm}\leq
    \frac{\lc\diam_{\bR^3}(E_i)\rc^3}{|E_i|} \int_0^T \int_{E_i\times E_i}
    \frac{|U u_0(t,y)-U u_0(t,x)|^2}{|x-y|^3}\,\diff{(\sigma\otimes\sigma)(x,y)}\diff{t}.
\end{align*}
We now use assumption \eqref{eq:ass_detectors} and \eqref{eq:fond_prop1_ass_det} to deduce that
\begin{align*}
    &\int_0^T \int_{E_i} \bigg| 
    U u_0(t,y)-\frac{1}{|E_i|}\int_{E_i} U u_0(t,x)\,\diff{\sigma(x)} \bigg|^2\,\diff{\sigma(y)}\diff{t} \\ &\hspace{5mm}\leq C_{\mathrm{ecc}}^2 \mu 2^{-j_0} \int_0^T \int_{\Sigma\times \Sigma} \mathbbm{1}_{E_i\times E_i}(x,y)
    \frac{|U u_0(t,y)-U u_0(t,x)|^2}{|x-y|^3}\,\diff{(\sigma\otimes\sigma)(x,y)}\diff{t}.
\end{align*}
Notice that the following inequality holds $(\diff{\sigma}\otimes\diff{\sigma})$-a.e.:
\begin{align*}
    \sum_{i=1}^N \mathbbm{1}_{E_i\times E_i} \leq \mathbbm{1}_{\Sigma\times\Sigma}.
\end{align*}
Therefore, summing over $i$, we get
\begin{align*}
    &\|U u_0 - P_{\cF}U u_0\|_{L^2(0,T;L^2(\Sigma))}^2 \\ &\hspace{5mm}\leq
    C_{\mathrm{ecc}}^2 \mu 2^{-j_0} \int_0^T \int_{\Sigma\times \Sigma}
    \frac{|U u_0(t,y)-U u_0(t,x)|^2}{|x-y|^3}\,\diff{(\sigma\otimes\sigma)(x,y)}\diff{t} \\ &\hspace{5mm}\leq
    C_{\mathrm{ecc}}^2\mu 2^{-j_0}\|U u_0\|_{L^2(0,T;H^{1/2}(\Sigma))}^2.
\end{align*}
Using the estimate from Lemma~\ref{lem:H12_stability}, we conclude that
\begin{align*}
    \|U u_0 - P_{\cF}U u_0\|_{L^2(0,T;L^2(\Sigma))}^2 &\leq 
    C_{\mathrm{ecc}}^2 C_1^2 \mu 2^{-j_0} \|u_0\|_{H^{1/2}(\bR^3)}^2.
\end{align*}
The Littlewood-Paley property \eqref{eq:litt_paley} implies that
\begin{align*}
    \|u_0\|_{H^{1/2}(\bR^3)}^2 \leq C'' 2^{j_0} \|u_0\|_{L^2(\bR^3)}^2.
\end{align*}
for some constant $C''>0$ depending only on the wavelet dictionary. We finally conclude that
\begin{align*}
    \|U u_0 - P_{\cF}U u_0\|_{L^2(0,T;L^2(\Sigma))}^2 \leq
    C_{\mathrm{ecc}}^{2} C_1^2 C'' \mu \|u_0\|_{L^2(\bR^3)}^2,
\end{align*}
which proves \eqref{eq:fond_est_prop1}.
\end{proof}
The coherence bounds in \eqref{eq:coherence_bounds} follow from the next Proposition.
\begin{proposition}
\label{prop:fond_prop2}
Consider the setting introduced in sections \ref{sec:wavelets}, \ref{sec:freespace}, \ref{sec:detectors}. Suppose that Assumption~\ref{ass:stability} is satisfied for some $K\subset\Omega_0\Subset\Omega_1\Subset\Omega$, where $K$ is as in \eqref{eq:Kdef}. Then there exists a constant $B_0\geq 1$, depending only on $T$, on the wavelet basis, on $\Omega_0,\Omega_1,\Omega$ and on $C_{\mathrm{ecc}}$, such that
\begin{align*}
    \|F_i \phi_{j,n}\|_{L^2(0,T)} \leq B_0 |E_i|^{1/2},\qquad i=1,\dots,N.
\end{align*}
In particular, \eqref{eq:coherence_bounds} is satisfied with $B=B_0 |\Sigma|^{1/2}$.
\end{proposition}

\begin{proof}[Proof of Proposition~\ref{prop:fond_prop2}]
Let $u_{j,n}\in C^1([0,T];H^1(\bR^3))$ denote the solution of \eqref{eq:waveIVP} corresponding to the initial datum $\phi_{j,n}\in C^2_c(\Omega_0)$; in other words, using the notation from Section~\ref{sec:freespace}, we have that $u_{j,n}\coloneqq V \phi_{j,n}$.

Notice from \eqref{eq:wav_def} that the following estimates hold, where the implicit constants depend only on the wavelet dictionary:
\begin{equation}
\label{eq:wav_scaling}
    \|\phi_{j,n}\|_{L^{\infty}(\bR^3)} \lesssim 2^{3/2 j},\qquad \|\nabla{\phi_{j,n}}\|_{L^{\infty}(\bR^3)} \lesssim 2^{5/2 j},\qquad (j,n)\in\Gamma.
\end{equation}
Moreover, there exists a constant $C>0$ such that, for every $(j,n)\in\Gamma$, there exists a ball $Q_{j,n}$ of radius $C2^{-j}$ such that
\begin{align*}
    \supp(\phi_{j,n})\subset Q_{j,n}.
\end{align*}
Let $x_{j,n}$ be the center of $Q_{j,n}$. We adopt the following notation for $x\in\bR^3$ and $t\geq 0$:
\begin{align*}
    d_{j,n}(x) \coloneqq \|x-x_{j,n}\|,\qquad
    D_{j,n}(x,t) \coloneqq \mathbbm{1}_{[d_{j,n}(x)-C2^{-j},d_{j,n}(x)+C2^{-j}]}(t).
\end{align*}
We now prove that, for $x\in\bR^3$ and $t\geq0$, the following holds for some constant $C'>0$ depending only on $C$:
\begin{equation}
\label{eq:estimate_bt_qjn}
    |\partial B_t(x)\cap Q_{j,n}| \leq
    C' 2^{-2j} D_{j,n}(x,t).
\end{equation}
Indeed, clearly $\partial B_t(x) \cap Q_{j,n} = \emptyset$ if $t<d_{j,n}(x)-C2^{-j}$ or $t>d_{j,n}(x)+C2^{-j}$, and thus $|\partial B_t(x) \cap Q_{j,n}| = 0$ in this case. Therefore, to prove \eqref{eq:estimate_bt_qjn} we just need to find an upper bound of $|\partial B_t(x) \cap Q_{j,n}|$ for $t\in[d_{j,n}(x)-C2^{-j},d_{j,n}(x)+C2^{-j}]$. We distinguish two cases.
\begin{figure}[t!]
    \centering
    \includegraphics[width=0.4\textwidth]{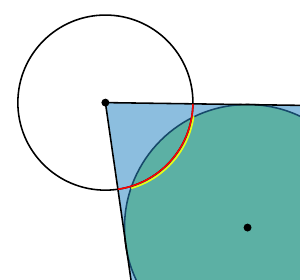}
    \put(-119,106){\scalebox{0.8}{$x$}}
    \put(-103,86){\scalebox{0.9}{$C_{j,n}$}}
    \put(-85,15){\scalebox{0.9}{$Q_{j,n}$}}
    \put(-30,36){\scalebox{0.8}{$x_{j,n}$}}
    \put(-72,140){\scalebox{0.9}{$\partial B_t(x)$}}
    \caption{The sets in the proof of Proposition~\ref{prop:fond_prop2}. In red, the set $\partial B_t(x)\cap C_{j,n}$; in yellow, the set $\partial B_t(x)\cap Q_{j,n}$.}
    \label{fig:balls}
\end{figure}
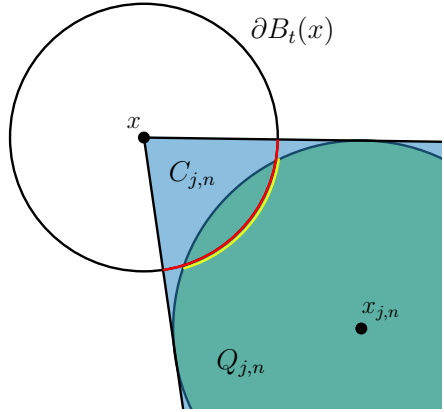
\begin{enumerate}
    \item Suppose that $d_{j,n}(x)\leq C2^{-j}$. Then, for $t\leq d_{j,n}(x)+C2^{-j}\leq 2C2^{-j}$,
    \begin{align*}
        |\partial B_t(x)\cap Q_{j,n}| \leq |\partial B_t(x)| = 4\pi t^2 \leq 16\pi C^2 2^{-2j}.
    \end{align*}
    \item Suppose that $d_{j,n}(x)>C2^{-j}$. Let $C_{j,n}$ be the simple cone -- considered as a $3$ dimensional solid -- with vertex $x$ whose boundary is tangent to $Q_{j,n}$ -- see Fig.~\ref{fig:balls}. Then the cross-section of $C_{j,n}$ subtends an angle $2\theta$, where
    \begin{align*}
        \sin{\theta} = \frac{C 2^{-j}}{d_{j,n}(x)}.
    \end{align*}
    The set $\partial B_t (x)\cap C_{j,n}$ is a spherical cap, whose area is
    \begin{align*}
        |\partial B_t (x)\cap C_{j,n}| = 2\pi(1-\cos{\theta})t^2 =
        2\pi\lc 1-\sqrt{1-\frac{C^2 2^{-2j}}{d_{j,n}(x)^2}} \rc t^2.
    \end{align*}
    As $1-\sqrt{1-x}\leq x$, we get that, for $t\leq d_{j,n}(x)+C2^{-j}\leq 2d_{j,n}(x)$,
    \begin{align*}
        |\partial B_t(x)\cap Q_{j,n}| &\leq
        |\partial B_t(x)\cap C_{j,n}| \leq
        2\pi \frac{C^2 2^{-2j}}{d_{j,n}(x)^2} t^2 \leq
        8\pi C^2 2^{-2j}.
    \end{align*}
\end{enumerate}
This proves \eqref{eq:estimate_bt_qjn}. Equation \eqref{eq:estimate_bt_qjn} and \eqref{eq:wav_scaling} imply the following estimates:
\begin{align*}
    &\int_{\partial B_t(x)} |\phi_{j,n}(y)|\,\diff{\sigma(y)} =
    \int_{\partial B_t(x)\cap Q_{j,n}} |\phi_{j,n}(y)|\,\diff{\sigma(y)} \leq
    C'' 2^{-j/2} D_{j,n}(x,t),\\
    &\int_{\partial B_t(x)} |\nabla{\phi_{j,n}}(y)|\,\diff{\sigma(y)} =
    \int_{\partial B_t(x)\cap Q_{j,n}} |\nabla{\phi_{j,n}}(y)|\,\diff{\sigma(y)} \leq
    C'' 2^{j/2} D_{j,n}(x,t),
\end{align*}
where $C''$ depends only on the wavelet dictionary and on $C'$.

For $t<d(K,\Sigma)$, we get that $u_{j,n}(\cdot,t)=0$ on $\Sigma$; for $t\geq d(K,\Sigma)$ and $x\in\Sigma$, using the explicit formula \eqref{eq:kirchoff} for the solution of \eqref{eq:waveIVP}, we get that
\begin{align*}
    |u_{j,n}(x,t)| &\leq
    \frac{1}{4\pi t^2} \int_{\partial B_t(x)} \lc|\phi_{j,n}(y)|+|\nabla{\phi_{j,n}}(y)| |y-x|\rc\,\diff{\sigma(y)} \\ &\leq
    \frac{1}{4\pi t^2} \lc \int_{\partial B_t(x)}|\phi_{j,n}(y)|\,\diff{\sigma(y)}+
    \diam_{\bR^3}(\partial B_t(x))\int_{\partial B_t(x)} |\nabla{\phi_{j,n}(y)}|\,\diff{\sigma(y)}\rc \\ &\leq
    \frac{1}{4\pi t^2} \lc 1+\diam_{\bR^3}(\partial B_t(x)) \rc C'' 2^{j/2} D_{j,n}(x,t) \\ &\leq
    \frac{1+2T}{4\pi d(K,\Sigma)^2} C'' 2^{j/2} D_{j,n}(x,t) \eqqcolon C''' 2^{j/2} D_{j,n}(x,t).
\end{align*}
We deduce that
\begin{align*}
    \|F_i \phi_{j,n}\|_{L^2(0,T)}^2 &=
    \int_0^T \bigg| \frac{1}{|E_i|^{1/2}}\int_{E_i} u_{j,n}(x,t)\,\diff{\sigma(x)} \bigg|^2\,\diff{t} \\ &\leq
    \frac{(C''')^2 2^{j}}{|E_i|} \int_0^T \bigg| \int_{E_i}
    D_{j,n}(x,t)\,\diff{\sigma(x)} \bigg|^2\,\diff{t} \\ &\leq
    (C''')^2 2^{j} \int_{E_i} \int_0^T |D_{j,n}(x,t)|^2\,\diff{t}\diff{\sigma(x)} \leq
    2(C''')^2 C |E_i|,
\end{align*}
and we conclude.
\end{proof}

\begin{proof}[Proof of Theorem~\ref{thm:main_result}]
The idea is to apply Theorem~\ref{thm:split2}. We have already argued that conditions $(a1)$, $(a2)$ and $(b)$ are satisfied thanks to Assumption~\ref{ass:stability}, Proposition~\ref{prop:fond_prop1} and Proposition~\ref{prop:fond_prop2}, respectively.

We now establish a correspondence between the measurements and the noise in the setting of Theorem~\ref{thm:main_result} with that of Theorem~\ref{thm:split2}. In Theorem~\ref{thm:split2}, the measurements are given by
\begin{align*}
    \tilde{y}_k = F_{t_k} u_0 + \tilde{\varepsilon}_k,
\end{align*}
where $F_{t_k}\in\cL(\cH_1,\cH_2)$ is the measurement operator corresponding to the sample $t_k\in\cD$ and $\tilde{\varepsilon}_k\in\cH_2$ is an additive noise. In Theorem~\ref{thm:main_result}, instead, the measurements are given by
\begin{align*}
    y_k(t) &= \frac{1}{|E_{i_k}|} \int_{E_{i_k}}U u_0(y,t)\,\diff{\sigma}(y)+\varepsilon_k(t) \\ &=
    |E_{i_k}|^{-1/2} \lc F_{i_k,T} u_0 + |E_{i_k}|^{1/2}\varepsilon_k(t) \rc.
\end{align*}
Therefore the measurements and the noise in the two settings are related by the identities
\begin{equation}
\label{eq:equivalence_measurements}
    y_k = |E_{i_k}|^{-1/2} \tilde{y_k},\qquad
    \tilde{\varepsilon}_k = |E_{i_k}|^{1/2}\varepsilon_k.
\end{equation}
Given this correspondence, we can now verify the assumptions on the noise and on the truncation errors. If $\|\varepsilon_k\|_{L^2(0,T)}\leq \beta$, then $\|\tilde{\varepsilon}_k\|_{L^2(0,T)}\leq |E_{i_k}|^{1/2}\beta$; therefore, condition \eqref{eq:noise-ass} is satisfied in our case with $\tilde{\beta}=|\Sigma|^{1/2}\beta$.

The assumptions on the truncation error are given by \eqref{eq:truncation-error}. In our case, the measurement operators are uniformly bounded, as they are given by $F_i =f_{E_i} U$ where $f_{E_i}$ is the operator of unit norm defined by
\begin{align*}
    f_{E_i}(v) \coloneqq \frac{1}{|E_i|^{1/2}} \int_{E_i} v,\qquad v\in L^2(0,T;L^2(\Sigma)).
\end{align*}
Therefore, we get that, if $\|P_{\leq j_0}^{\perp} u_0\|_{\cH_1}\leq r$, then
\begin{align*}
    \sup_{i=1,\dots,N} \nu_i^{-1/2} \|F_i  P_{\leq j_0}^{\perp} u_0\|_{L^2(0,T)} \leq 
    |\Sigma|^{1/2} (\min_i |E_i|)^{-1/2} \|U\| r.
\end{align*}
We conclude that \eqref{eq:truncation-error} is satisfied with $\tilde{r}=\max\{1,|\Sigma|^{1/2} \|U\| (\min_i |E_i|)^{-1/2}\}r$. Notice that $\|U\|$ depends only on $T$ and on the constants $c,C$ in Assumption~\ref{ass:stability}.

Finally, notice that in Theorem~\ref{thm:split2} the minimization problem to be solved is \eqref{eq:minimization-split2}. In the PAT case, $\nu_{i_k}^{-1}=|\Sigma| |E_{i_k}|^{-1}$; therefore, the identity \eqref{eq:equivalence_measurements} implies that this problem is equivalent in our setting to
\begin{align*}
    \min_{u\in\cM_{\leq j_0}} \|\Phi{u}\|_1 \colon\quad
    \frac{1}{m}\sum_{k=1}^m \bigg\| \frac{1}{|E_{i_k}|}\int_{E_{i_k}}U u(y,\cdot)\,\diff{\sigma}(y)-y_k \bigg\|_{L^2(0,T)}^2 \leq (\beta+C_3 \tilde{r})^2,
\end{align*}
where $C_3= C_3'/|\Sigma|^{1/2}$, which is exactly \eqref{eq:minimization_problem}.

Having established the correspondence of Theorem~\ref{thm:split2} with Theorem~\ref{thm:main_result} in our setting, the result is proved.
\end{proof}

\subsection{Proof for general sensing patterns}
\label{sec:proof-new-result}

In this section, we provide a proof of Theorem~\ref{thm:new-result}. Essentially, it follows along the lines of the proof of Theorem~\ref{thm:main_result}. The idea is always to apply Theorem~\ref{thm:split2} and verify that conditions $(a1),(a2),(b)$ stated in Sec.~\ref{sec:proof-main-result} hold for suitable measurement operators and forward maps.

In this case, we define the measurement operators as
\begin{align*}
    \tilde{F}_{i,T} u_0 \coloneqq \frac{1}{\sqrt{N}} P_i AMU u_0.
\end{align*}
The corresponding forward map is then given by
\begin{align*}
    \tilde{F}_T u_0 \coloneqq \frac{1}{\sqrt{N}} AMU u_0.
\end{align*}
Using the fact that $A$ is a unitary matrix, we get that
\begin{align*}
    \|\tilde{F}_T u_0\|_{(L^2(0,T))^N}^2 &= \|\tilde{F}_T u_0\|_{L^2(0,T;\ell^2(\bC^N))}^2 = 
    \frac{1}{N}\int_0^T |AMU u_0(t)|^2 \,\diff{t} \\ &=
    \frac{1}{N}\int_0^T |MU u_0(t)|^2 \,\diff{t} =
    \frac{1}{N}\|MU u_0\|_{(L^2(0,T))^N}^2.
\end{align*}
Explicitly, the operator on the RHS is given by
\begin{align*}
    \lc \frac{1}{\sqrt{N}} MU u_0(t) \rc_i = 
    \frac{1}{\sqrt{N|E_i|}} \frac{1}{\sqrt{|E_i|}} \int_{E_i} U u_0(t,y)\,\diff{\sigma}(y) = \frac{1}{\sqrt{N|E_i|}} F_i u_0(t).
\end{align*}
Using \eqref{eq:quasi-unif}, we get that, up to an isomorphism, $\tfrac{1}{\sqrt{N}} MU u_0(t)$ is given by the forward map from Sec.~\ref{sec:proof-main-result}. Therefore, we conclude that condition $(a)$ in this case follows as in the case of Theorem~\ref{thm:main_result}.

Moreover, condition $(b)$ follows directly from \eqref{eq:new-coherence}, with $\nu_i\equiv 1/N$. The new measurements $\tilde{y}_k$ and noise $\tilde{\varepsilon}_k$ are given by
\begin{align*}
    \tilde{y}_k = N^{-1/2}y_k,\quad \tilde{\varepsilon}_k = N^{-1/2}\varepsilon_k.
\end{align*}
Moreover, the truncation $u_R$ satisfies
\begin{align*}
    \max_{i=1,\dots,N} \|P_i AMU u_R\|_{L^2(0,T)} \leq
    \|MUu_R\|_{L^2(0,T)} \lesssim \sqrt{N}\|u_R\|_{L^2(\bR^3)} \leq \sqrt{N}r,
\end{align*}
and thus we can choose $\tilde{r}=C\sqrt{N}r$. Concerning the noise bound, we get
\begin{align*}
    \nu_{i_k}^{-1/2} \|\tilde{\varepsilon}_k\|_{L^2(0,T)} = \|\varepsilon_k\|_{L^2(0,T)} \leq \beta.
\end{align*}
Finally, the two minimization problems are related via the identity
\begin{align*}
    \nu_{i_k}^{-1} \|\tilde{F}_{i_k}v-\tilde{y}_k\|_{L^2(0,T)}^2 =
    \|P_{i_k}AMUv-y_k\|_{L^2(0,T)}^2.
\end{align*}
This verifies the assumptions on the noise and on the truncation error, leading to the validation of the hypotheses of Theorem~\ref{thm:split2}, which can then be applied to prove Theorem~\ref{thm:new-result}.

\section*{Acknowledgments}

Co-funded by the European Union (ERC, SAMPDE, 101041040, and Next Generation EU). Views and opinions expressed are however those of the authors only and do not necessarily reflect those of the European Union or the European Research Council. Neither the European Union nor the granting authority can be held responsible for them. The author is member of the ``Gruppo Nazionale per l’Analisi Matematica, la Probabilità e le loro Applicazioni'', of the ``Istituto Nazionale di Alta Matematica''. The research was supported in part by the MIUR Excellence Department Project awarded to Dipartimento di Matematica, Università di Genova, CUP D33C23001110001.

I would also like to thank Prof.\ Giovanni S.\ Alberti and Prof.\ Matteo Santacesaria, for introducing me to photoacoustic tomography and for their careful revision of the manuscript, and Prof.\ Rakesh, for the insightful discussion regarding stability estimates for inverse problems in hyperbolic PDEs.

\section*{Data statement}
No datasets were generated or analyzed during the current study.

\section*{Conflict of interest}
The author declares that he has no conflict of interest.

\bibliography{refs}{}

\begin{thebibliography}{10}

\bibitem{acosta2015}
Sebasti{\'a}n Acosta and Carlos Montalto.
\newblock Multiwave imaging in an enclosure with variable wave speed.
\newblock {\em Inverse Problems}, 31, 2015.

\bibitem{agranovsky2007}
M.~Agranovsky, P.~Kuchment, and L.~Kunyansky.
\newblock On reconstruction formulas and algorithms for the thermoacoustic
  tomography.
\newblock In {\em Photoacoustic imaging and spectroscopy}, pages 89--102. CRC
  press, 2009.

\bibitem{agranovsky2}
Mark Agranovsky and Peter Kuchment.
\newblock Uniqueness of reconstruction and an inversion procedure for
  thermoacoustic and photoacoustic tomography with variable sound speed.
\newblock {\em Inverse Problems}, 23:2089 -- 2102, 2007.

\bibitem{alberti2020}
Giovanni Alberti, Yves Capdeboscq, and Yannick Privat.
\newblock On the randomised stability constant for inverse problems.
\newblock {\em Mathematics in Engineering}, 2:264--286, 02 2020.

\bibitem{campodonico2021}
Giovanni~S. Alberti, Paolo Campodonico, and Matteo Santacesaria.
\newblock Compressed sensing photoacoustic tomography reduces to compressed
  sensing for undersampled fourier measurements.
\newblock {\em SIAM Journal on Imaging Sciences}, 14(3):1039--1077, 2021.

\bibitem{alberti2018}
Giovanni~S. Alberti and Yves Capdeboscq.
\newblock {\em Lectures on Elliptic Methods for Hybrid Inverse Problems},
  volume~25 of {\em Cours Sp{\'e}cialis{\'e}s}.
\newblock Soci{\'e}t{\'e} Math{\'e}matique de France, Paris, 2018.

\bibitem{split1}
Giovanni~S. Alberti, Alessandro Felisi, Matteo Santacesaria, and S.~Ivan
  Trapasso.
\newblock Compressed sensing for inverse problems and the sample complexity of
  the sparse {R}adon transform.
\newblock {\em Journal of the European Mathematical Society \emph{(to
  appear)}}.

\bibitem{split2}
Giovanni~S. Alberti, Alessandro Felisi, Matteo Santacesaria, and S.~Ivan
  Trapasso.
\newblock Compressed sensing for inverse problems {II}: applications to
  deconvolution, source recovery, and {MRI}.
\newblock {\em arXiv:2501.01929}, 2025.

\bibitem{alberti2021}
Giovanni~S. Alberti and Matteo Santacesaria.
\newblock Infinite dimensional compressed sensing from anisotropic measurements
  and applications to inverse problems in {PDE}.
\newblock {\em Applied and Computational Harmonic Analysis}, 50:105--146, 2021.

\bibitem{ammari2010}
Habib Ammari, Emmanuel Bossy, Vincent Jugnon, and Hyeonbae Kang.
\newblock Mathematical modeling in photoacoustic imaging of small absorbers.
\newblock {\em SIAM Review}, 52(4):677--695, 2010.

\bibitem{antholzer2019}
Stephan Antholzer, Johannes Schwab, Johnnes Bauer-Marschallinger, Peter
  Burgholzer, and Markus Haltmeier.
\newblock {NETT} regularization for compressed sensing photoacoustic
  tomography.
\newblock In {\em BiOS}, 2019.

\bibitem{arridge2016}
Simon Arridge, Paul Beard, Marta Betcke, Ben Cox, Nam Huynh, Felix Lucka,
  Olumide Ogunlade, and Edward Zhang.
\newblock Accelerated high-resolution photoacoustic tomography via compressed
  sensing.
\newblock {\em Physics in Medicine \& Biology}, 61(24):8908, dec 2016.

\bibitem{belhachmi2016}
Zakaria Belhachmi, Thomas Glatz, and Otmar Scherzer.
\newblock A direct method for photoacoustic tomography with inhomogeneous sound
  speed.
\newblock {\em Inverse Problems}, 32(4):045005, mar 2016.

\bibitem{bergh2012interpolation}
J.~Bergh and J.~L{\"o}fstr{\"o}m.
\newblock {\em Interpolation Spaces: An Introduction}.
\newblock Grundlehren der mathematischen Wissenschaften. Springer Berlin
  Heidelberg, 2012.

\bibitem{betcke2017}
Marta~M. Betcke, Ben~T. Cox, Nam Huynh, Edward~Z. Zhang, Paul~C. Beard, and
  Simon~R. Arridge.
\newblock Acoustic wave field reconstruction from compressed measurements with
  application in photoacoustic tomography.
\newblock {\em IEEE Transactions on Computational Imaging}, 3(4):710--721,
  2017.

\bibitem{burgholzer2016}
P.~Burgholzer, M.~Sandbichler, F.~Krahmer, T.~Berer, and M.~Haltmeier.
\newblock {Sparsifying transformations of photoacoustic signals enabling
  compressed sensing algorithms}.
\newblock In Alexander~A. Oraevsky and Lihong~V. Wang, editors, {\em Photons
  Plus Ultrasound: Imaging and Sensing 2016}, volume 9708, page 970828.
  International Society for Optics and Photonics, SPIE, 2016.

\bibitem{CRT}
E.J. Candes, J.~Romberg, and T.~Tao.
\newblock Robust uncertainty principles: exact signal reconstruction from
  highly incomplete frequency information.
\newblock {\em IEEE Transactions on Information Theory}, 52(2):489--509, 2006.

\bibitem{CD}
Emmanuel~J. Candès and David Donoho.
\newblock New tight frames of curvelets and optimal representations of objects
  with $c^2$ singularities.
\newblock {\em Communications on Pure and Applied Mathematics}, 57:219 -- 266,
  02 2004.

\bibitem{chervova2016}
Olga Chervova and Lauri Oksanen.
\newblock Time reversal method with stabilizing boundary conditions for
  photoacoustic tomography.
\newblock {\em Inverse Problems}, 32(12):125004, nov 2016.

\bibitem{cox2007}
Benjamin~T. Cox, Simon~Robert Arridge, and Paul~C. Beard.
\newblock Photoacoustic tomography with a limited-aperture planar sensor and a
  reverberant cavity.
\newblock {\em Inverse Problems}, 23:S95 -- S112, 2007.

\bibitem{donoho}
D.L. Donoho.
\newblock Compressed sensing.
\newblock {\em IEEE Transactions on Information Theory}, 52(4):1289--1306,
  2006.

\bibitem{duarte2008}
Marco~F. Duarte, Mark~A. Davenport, Dharmpal Takhar, Jason~N. Laska, Ting Sun,
  Kevin~F. Kelly, and Richard Baraniuk.
\newblock Single-pixel imaging via compressive sampling.
\newblock {\em IEEE Signal Process. Mag.}, 25:83--91, 2008.

\bibitem{Ev}
Lawrence~C. Evans.
\newblock {\em Partial Differential Equations}.
\newblock American Mathematical Society, 2nd edition, 2010.

\bibitem{rondi}
Alessandro Felisi and Luca Rondi.
\newblock Full discretization and regularization for the calderón problem.
\newblock {\em Journal of Differential Equations}, 410:513--577, 2024.

\bibitem{foucartrauhut}
Simon Foucart and Holger Rauhut.
\newblock A mathematical introduction to compressive sensing.
\newblock In {\em Applied and Numerical Harmonic Analysis}, 2013.

\bibitem{goda2009}
Keisuke Goda, Kevin.~K. Tsia, and Bahram Jalali.
\newblock Serial time-encoded amplified imaging for real-time observation of
  fast dynamic phenomena.
\newblock {\em Nature}, 458:1145--1149, 2009.

\bibitem{GL}
Kanghui Guo and Demetrio Labate.
\newblock Optimally sparse multidimensional representation using shearlets.
\newblock {\em SIAM J. Math. Anal.}, 39(1):298--318, 2007.

\bibitem{haltmeier2016}
Markus Haltmeier, Thomas Berer, Sunghwan Moon, and Peter Burgholzer.
\newblock Compressed sensing and sparsity in photoacoustic tomography.
\newblock {\em Journal of Optics}, 18, 2016.

\bibitem{haltmeier-nguyen16}
Markus Haltmeier and Linh~V. Nguyen.
\newblock Analysis of iterative methods in photoacoustic tomography with
  variable sound speed.
\newblock {\em SIAM J. Imaging Sci.}, 10:751--781, 2017.

\bibitem{haltmeier2017}
Markus Haltmeier, Michael Sandbichler, Thomas Berer, Johannes
  Bauer-Marschallinger, Peter Burgholzer, and Linh~V. Nguyen.
\newblock A sparsification and reconstruction strategy for compressed sensing
  photoacoustic tomography.
\newblock {\em The Journal of the Acoustical Society of America}, 143 6:3838,
  2017.

\bibitem{haltmeier2024}
Markus Haltmeier, Matthias Ye, Karoline Felbermayer, Florian Hinterleitner, and
  Peter Burgholzer.
\newblock Design, implementation, and analysis of a compressed sensing
  photoacoustic projection imaging system.
\newblock {\em Journal of Biomedical Optics}, 29, 2024.

\bibitem{holman2015}
Benjamin~Robert Holman and Leonid~A. Kunyansky.
\newblock Gradual time reversal in thermo- and photo-acoustic tomography within
  a resonant cavity.
\newblock {\em Inverse Problems}, 31, 2014.

\bibitem{huynh2016}
Nam Huynh, Edward~Z. Zhang, Marta~M. Betcke, Simon~Robert Arridge, Paul~C.
  Beard, and Benjamin~T. Cox.
\newblock Single-pixel optical camera for video rate ultrasonic imaging.
\newblock {\em Optica}, 2016.

\bibitem{hyvonen}
Nuutti Hyv{\"o}nen.
\newblock Approximating idealized boundary data of electric impedance
  tomography by electrode measurements.
\newblock {\em Mathematical Models and Methods in Applied Sciences},
  19:1185--1202, 2009.

\bibitem{kuchment2007}
Peter Kuchment and Leonid~A. Kunyansky.
\newblock Mathematics of thermoacoustic tomography.
\newblock {\em European Journal of Applied Mathematics}, 19:191 -- 224, 2007.

\bibitem{kunyansky2013}
Leonid~A. Kunyansky, Benjamin~Robert Holman, and Benjamin~T. Cox.
\newblock Photoacoustic tomography in a rectangular reflecting cavity.
\newblock {\em Inverse Problems}, 29, 2013.

\bibitem{kuty_cartoon}
Gitta Kutyniok and Wang-Q Lim.
\newblock Compactly supported shearlets are optimally sparse.
\newblock {\em J. Approx. Theory}, 163(11):1564--1589, 2011.

\bibitem{Ma}
Stéphane Mallat.
\newblock {\em A Wavelet Tour of Signal Processing. {T}he Sparse Way.}
\newblock Elsevier/Academic Press, Amsterdam, 3rd edition, 2009.

\bibitem{mathison2019}
Chase Mathison.
\newblock Sampling in thermoacoustic tomography.
\newblock {\em Journal of Inverse and Ill-posed Problems}, 28:881 -- 897, 2019.

\bibitem{meng2012}
Jing Meng, Dong Liang, and Liang Song.
\newblock Compressed sensing photoacoustic tomography in vivo in time and
  frequency domains.
\newblock In {\em Proceedings of 2012 IEEE-EMBS International Conference on
  Biomedical and Health Informatics}, pages 717--720, 2012.

\bibitem{Me}
Yves Meyer.
\newblock {\em Wavelets and Operators}.
\newblock Cambridge University Press, Cambridge, 1992.

\bibitem{valdinoci}
Eleonora~Di Nezza, Giampiero Palatucci, and Enrico Valdinoci.
\newblock Hitchhiker's guide to the fractional sobolev spaces.
\newblock {\em Bulletin Des Sciences Mathematiques}, 136:521--573, 2011.

\bibitem{nguyen2011}
Linh~V Nguyen.
\newblock On singularities and instability of reconstruction in thermoacoustic
  tomography.
\newblock {\em Contemporary Mathematics}, 559, 2011.

\bibitem{provost2009}
Jean Provost and Frédéric Lesage.
\newblock The application of compressed sensing for photo-acoustic tomography.
\newblock {\em IEEE Transactions on Medical Imaging}, 28(4):585--594, 2009.

\bibitem{sandbichler2015}
M.~Sandbichler, F.~Krahmer, T.~Berer, P.~Burgholzer, and M.~Haltmeier.
\newblock A novel compressed sensing scheme for photoacoustic tomography.
\newblock {\em SIAM Journal on Applied Mathematics}, 75(6):2475--2494, 2015.

\bibitem{stefanov2009}
Plamen Stefanov and Gunther Uhlmann.
\newblock Thermoacoustic tomography with variable sound speed.
\newblock {\em Inverse Problems}, 25:075011, 2009.

\bibitem{stefanov2015}
Plamen Stefanov and Yang Yang.
\newblock Multiwave tomography in a closed domain: averaged sharp time
  reversal.
\newblock {\em Inverse Problems}, 31, 2014.

\bibitem{stefanov2016}
Plamen Stefanov and Yang Yang.
\newblock Multiwave tomography with reflectors: Landweber's iteration.
\newblock {\em Inverse Problems \& Imaging}, 11(2), 2017.

\bibitem{tick2020}
Jenni Tick, Aki Pulkkinen, and Tanja Tarvainen.
\newblock Modelling of errors due to speed of sound variations in photoacoustic
  tomography using a {B}ayesian framework.
\newblock {\em Biomedical Physics \& Engineering Express}, 6, 2019.

\bibitem{tuchin2016}
Valery~V. Tuchin.
\newblock Handbook of optical biomedical diagnostics, second edition, volume 2:
  Methods, 2016.

\bibitem{wang2009}
Lihong~V. Wang.
\newblock Photoacoustic imaging and spectroscopy.
\newblock {\em Journal of Biomedical Optics}, 15:059901, 2009.

\bibitem{wang2007}
Lihong~V. Wang and Hsin‐I Wu.
\newblock Biomedical optics: Principles and imaging, 2008.

\bibitem{Zhang2008}
Edward~Z. Zhang, Jan Laufer, and Paul~C. Beard.
\newblock Backward-mode multiwavelength photoacoustic scanner using a planar
  fabry-perot polymer film ultrasound sensor for high-resolution
  three-dimensional imaging of biological tissues.
\newblock {\em Applied optics}, 47 4:561--77, 2008.

\end{thebibliography}
\bibliographystyle{plain}

\appendix

\section{Interpolation estimate}
\label{appendix:stab_est}
In this section, we provide a proof of the interpolation estimate from Lemma~\ref{lem:H12_stability}.
\begin{proof}[Proof of Lemma~\ref{lem:H12_stability}]
The idea of the proof is to use the complex interpolation method to obtain stability with respect to $H^{1/2}$ norms using stability estimates in the $L^2$ and $H^1$ norms, respectively, from Assumption~\ref{ass:stability}.

We fix a finite atlas of $\Sigma$, given by $(V_i,\varphi_i)_{i=1}^n$, and a partition of unity $(\eta_i)_{i=1}^n\subset C^{\infty}(\Sigma)$ subordinate to it. From Assumption~\ref{ass:stability}, we get that, for $i=1,\dots,n$ and $u_0\in C^2_c(\Omega_1)$,
\begin{equation}
\label{eq:H1stability_eq1}
    \int_0^{T} \int_{\Sigma} |\eta_i U u_0(y,t)|^2\,\diff{\sigma(y)}\diff{t} \leq \|U u_0\|_{L^2(0,T;L^2(\Sigma))}^2 \leq C^2 \|u_0\|_{L^2(\bR^3)}^2,
\end{equation}
and also
\begin{equation}
\label{eq:H1stability_eq2}
\begin{split}
    &\int_0^{T} \int_{\Sigma} |\nabla_{\Sigma}\lc\eta_i U u_0\rc(y,t)|^2\,\diff{\sigma(y)}\diff{t} \\ &\hspace{5mm}\leq
    2\int_0^{T} \int_{\Sigma} |\nabla_{\Sigma}{\eta_i}|^2 |U u_0(y,t)|^2\,\diff{\sigma(y)}\diff{t} \\ &\hspace{10mm}+
    2\int_0^{T} \int_{\Sigma} |\nabla{V u_0(y,t)}|^2\,\diff{\sigma(y)}\diff{t} \\ &\hspace{5mm}\leq
    2 C^2\|\nabla_{\Sigma}{\eta_i}\|_{L^{\infty}(\bR^3)}^2 \|u_0\|_{L^2(\bR^3)}^2+
    2 C_1^2 \|u_0\|_{H^1(\bR^3)}^2 \\ &\hspace{5mm}\leq
    C_2^2 \|u_0\|_{H^1(\bR^3)}^2,
\end{split}
\end{equation}
where $C_2\coloneqq \sqrt{2}\big( C^2 \max_i\|\nabla_{\Sigma}{\eta_i}\|_{L^{\infty}(\bR^3)}^2 + C_1^2 \big)^{1/2}$; the meaning of $\nabla_{\bR^3}$ and $\nabla_{\Sigma}$ is the same as in Sec.~\ref{sec:spherical}. We now have that
\begin{align*}
    &\int_0^{T} \int_{\Sigma} |\eta_i U u_0(y,t)|^2\,\diff{\sigma(y)}\diff{t} \\ &\hspace{5mm}=
    \int_0^T \int_{\varphi_i(V_i)} |\eta_i U u_0 (\varphi_i^{-1}(x),t)|^2 \sqrt{\det\lc (D\varphi_i^{-1})^{\top}(D\varphi_i^{-1})\rc}\,\diff{x}\diff{t} \\ &\hspace{5mm}\geq
    C_3^2 \int_0^T \int_{\varphi_i(V_i)} |\eta_i U u_0 (\varphi_i^{-1}(x),t)|^2\,\diff{x}\diff{t} \\ &\hspace{5mm}=
    C_3^2 \|(\eta_i U u_0)\circ(\varphi_i^{-1}\otimes I)\|_{L^2(0,T;L^2(\bR^2))}^2.
\end{align*}
where $C_3^2\coloneqq \min_{ \substack{x\in\varphi_i(V_i) \\ i=1,\dots,n}} \sqrt{\det\lc (D\varphi_i^{-1})^{\top}(D\varphi_i^{-1})\rc}$. This, together with \eqref{eq:H1stability_eq1}, implies that, for $i=1,\dots,n$ and $u_0\in C^2_c(\Omega_1)$,
\begin{equation}
\label{eq:H1stability_eq3}
    \|(\eta_i U u_0)\circ(\varphi_i^{-1}\otimes I)\|_{L^2(0,T;L^2(\bR^2))} \leq C_4 \|u_0\|_{L^2(\bR^3)},
\end{equation}
where $C_4\coloneqq C/C_3$.

Analogously, using \eqref{eq:H1stability_eq2}, we get that, for $i=1,\dots,n$ and $u_0\in C^2_c(\Omega_1)$,
\begin{align*}
    \|\big(\nabla_{\Sigma}(\eta_i U u_0)\big)\circ(\varphi_i^{-1}\otimes I)\|_{L^2(0,T;L^2(\bR^2))} \leq C_5 \|u_0\|_{H^1(\bR^3)},
\end{align*}
where $C_5\coloneqq C_2/C_3$. For $v\in C^2_c(V_i)$, we have that
\begin{align*}
    |\nabla (v\circ\varphi_i^{-1})(x)| = 
    |(D\varphi_i^{-1})^{\top}(x)(\nabla_{\Sigma}v)(\varphi_i^{-1}(x))| \leq
    C_6 |(\nabla_{\Sigma}v)(\varphi_i^{-1}(x))|,
\end{align*}
where $C_6\coloneqq \max_{ \substack{x\in\varphi_i(V_i) \\ i=1,\dots,n}} \|D\varphi_i^{-1}(x)\|$; here the matrix norm $\|\cdot\|$ refers to the operator norm on $\ell^2(\bR^2)$. This implies that, for $i=1,\dots,n$ and $u_0\in C^2_c(\Omega_1)$,
\begin{equation}
\label{eq:H1stability_eq4}
\begin{split}
    &\|(\eta_i U u_0)\circ(\varphi_i^{-1}\otimes I)\|_{L^2(0,T;H^1(\bR^2))}^2 \\ &\hspace{5mm}=
    \|(\eta_i U u_0)\circ(\varphi_i^{-1}\otimes I)\|_{L^2(0,T;L^2(\bR^2))}^2 \\ &\hspace{10mm}+
    \|\nabla\big((\eta_i U u_0)\circ(\varphi_i^{-1}\otimes I)\big)\|_{L^2(0,T;L^2(\bR^2))}^2 \\ &\hspace{5mm}\leq
    C_4^2 \|u_0\|_{L^2(\bR^3)}^2 + C_5^2 C_6^2 \|u_0\|_{H^1(\bR^3)}^2 \leq
    C_7^2 \|u_0\|_{H^1(\bR^3)}^2.
\end{split}
\end{equation}
where $C_7\coloneqq \big( C_4^2+C_5^2 C_6^2 \big)^{1/2}$.

Let $\eta\in C^{\infty}_c(\bR^3)$ be a smooth cutoff such that $\eta\in[0,1]$, $\eta|_{\Omega_0}\equiv 1$ and $\supp(\eta)\subset \Omega_1$. Given $u_0\in C^{2}(\bR^3)$, we have that $\eta u_0\in C^2_c(\Omega_1)$. For every $i=1,\dots,n$, we consider the following linear operators
\begin{align*}
    T_i u_0 \coloneqq \big(\eta_i U (\eta u_0)\big)\circ(\varphi_i^{-1}\otimes I),
\end{align*}
defined for $u_0\in C^2_c(\bR^3)$. The stability estimates obtained in \eqref{eq:H1stability_eq3} and \eqref{eq:H1stability_eq4} imply that
\begin{align*}
    &\|T_i u_0\|_{L^2(0,T;L^2(\bR^2))} \leq C_4 \|\eta u_0\|_{L^2(\bR^3)} \leq C_4 \|u_0\|_{L^2(\bR^3)},\\
    &\|T_i u_0\|_{L^2(0,T;H^1(\bR^2))} \leq C_7 \|\eta u_0\|_{H^1(\bR^3)} \leq C_7 C_8 \|u_0\|_{H^1(\bR^3)},
\end{align*}
where $C_8$ depends only on $\eta$.\footnote{Notice that it would not have been possible to obtain the bound $\|\eta u_0\|_{H^1(\bR^3)}\lesssim \|u_0\|_{H^1(\bR^3)}$ for a non-smooth cutoff $\eta$; this motivates the necessity of considering the slightly bigger set $\Omega_1$.} These estimates imply that $T_i$ can be extended by density and continuity to a bounded linear operator $T_i\colon L^2(\bR^3)\rightarrow L^2(0,T;L^2(\bR^2))$ such that $T_i(H^1(\bR^3))\subset L^2(0,T;H^1(\bR^2))$.

We can now apply the complex interpolation method \cite[Theorem 4.1.2, Theorem 6.4.5]{bergh2012interpolation} to conclude that
\begin{align*}
    \|T_i u_0\|_{L^2(0,T;H^{1/2}(\bR^2))} \leq C_9 \|u_0\|_{H^{1/2}(\bR^3)},
\end{align*}
where $C_9$ can be chosen, for instance, to be $\max(C_4,C_7 C_8)$. This implies that, for $u_0\in C^2_c(\Omega_0)$,
\begin{align*}
    \|\big(\eta_i U u_0\big)\circ(\varphi_i^{-1}\otimes I)\|_{L^2(0,T;H^{1/2}(\bR^2))} \leq C_9 \|u_0\|_{H^{1/2}(\bR^3)},
\end{align*}
as in this case $\eta u_0 = u_0$.

By definition of the $H^{1/2}(\Sigma)$ norm, there exists a constant $C_{10}>0$, depending only on the atlas of $\Sigma$, such that, for $i=1,\dots,n$ and for $u_0\in C_c^2(\Omega_0)$,
\begin{align*}
     \|\eta_i U u_0\|_{L^2(0,T;H^{1/2}(\Sigma))} &\leq C_{10}\|\big(\eta_i U u_0\big)\circ(\varphi_i^{-1}\otimes I)\|_{L^2(0,T;H^{1/2}(\bR^2))} \\ &\leq
     C_9 C_{10} \|u_0\|_{H^{1/2}(\bR^3)}.
\end{align*}
We conclude that, for $u_0\in C^2_c(\Omega_0)$,
\begin{align*}
    \|U u_0\|_{L^2(0,T;H^{1/2}(\Sigma))} \leq
    \sum_{i=1}^n \|\eta_i U u_0\|_{L^2(0,T;H^{1/2}(\Sigma))} \leq
    nC_9 C_{10} \|u_0\|_{H^{1/2}(\bR^3)}.
\end{align*}
This concludes the proof.
\end{proof}

\end{document}